\documentclass[hyphens]{amsart}

\usepackage{amsmath}
\usepackage{amsthm}
\usepackage{amssymb}
\usepackage[foot]{amsaddr}
\usepackage[backend=biber,style=alphabetic]{biblatex}
\usepackage{enumitem}
\usepackage{hyperref}
\usepackage[hyphenbreaks]{breakurl}
\usepackage{mathtools}
\usepackage[whole]{bxcjkjatype}

\newtheorem{thm}{Theorem}
\newtheorem{lem}[thm]{Lemma}
\newtheorem{prop}[thm]{Proposition}
\newtheorem{cor}[thm]{Corollary}
\newtheorem{question}{Question}

\theoremstyle{remark}
\newtheorem{rem}{Remark}

\theoremstyle{definition}
\newtheorem{definition}[thm]{Definition}

\newcommand{\co}{\operatorname{co}}

\include{macros}

\renewcommand{\seq}[1]{{#1}^{<\omega}}
\newcommand{\ns}[2]{{\mathrm{NS}_{#1}^{#2}}}
\newcommand{\nso}[1]{\ns{#1}{\omega}}
\newcommand{\limo}[1]{\mathrm{Lim}_{#1}^\omega}
\newcommand{\intprop}{\varphi_{\cap}}

\newcommand{\pinf}[1][]{\power_{#1}^\infty}
\newcommand{\intersects}[1]{\Hcal_{#1}}
\newcommand{\intfin}[1][]{\intersects{#1}^*}

\newenvironment{myenumerate}{\begin{enumerate}[label=\textup{(\alph*)}]}{\end{enumerate}}

\title{A Measure Theoretic Proof of $\p=\t$}

\author[James Hirschorn]{James Hirschorn}
\email{James.Hirschorn@quantitative-technologies.com}
\address{Quantitative Technologies\\ Phnom Penh, Kingdom of Cambodia}

\date{May 27, 2022}

\bibliography{bibliography}
\AtEveryBibitem{
  \clearlist{language}
  \clearfield{urlyear}
  \clearfield{issn}
}

\begin{document}
\maketitle

\begin{abstract}
  Rothberger's question of whether the two cardinals $\p$ and $\t$ are
  equal, posed back in 1948, was only answered fairly recently in the affirmative.
  Here we answer the more difficult progenitor question (posed in the same place)
  on when filter-bases can be refined to towers. Of equal import,
  our proof that $\p=\t$ addresses issues raised by Gowers concerning
  the ``proof path'' of the original solution.
  This is applied to obtain a gap spectrum result for the ultrapowers $\irr\div\U$.
\end{abstract}

\section{Introduction}

Rothberger~\cite[Lemma 7]{rothberger1948some} proved that if $\p=\aleph_1$ then $\p=\t$.
This begged the
question of whether this generalizes to $\p=\t$ without any constraint on the cardinality of $\p$?
This question of whether $\p$ and $\t$ are equal became the most important open problem in the
subfield of combinatorial set theory of the reals
(e.g.~\cite{malliaris2013general}, \cite{quanta-measure-infinity}).
It was recently solved, after more than sixty years, by
 Malliaris--Shelah~\cite{MalliarisShelah2016} with an affirmative answer.

The preceding Lemma~\cite[Lemma 6]{rothberger1948some} establishes that if $\p=\aleph_1$
then $\b>\p$\footnote{The stated hypothesis is that there are no $(\omega_1,\omega^*)$ gaps in
  $\pnfin$, and $\b$ is the smallest ordinal $\theta$ for which there is a $(\theta,\omega^*)$-gap
  as explained in~\cite[p.~34]{rothberger1948some} (see also \S\ref{sec:gaps}).} implies that
every filter-base of cardinality $\p$ can be refined to a tower. This begged the question of
whether this generalizes from $\aleph_1$ to arbitrary values of $\p$:

\begin{question}[Rothberger, 1948]
  \label{ques:rothberger}
  If\/ $\b>\p$ can every filter-base of cardinality\/ $\p$ be refined to a tower?
\end{question}

\cite[Lemma 7]{rothberger1948some} is an immediate consequence of \cite[Lemma
6]{rothberger1948some},
and the fact~\cite[Theorem~3]{rothberger1948some}
that $\t\le\b$. Thus the first attempt anyone would make to prove $\p=\t$, is to try and generalize the
proof of Lemma~6. In Rothberger's proof, as $\p=\aleph_1$ is assumed, only ordinals of countable
cofinality need to be handled in the induction step, whereas the challenge is to carry the induction
step to uncountable limits. This was undoubtedly tried by Rothberger himself, and is explicitly discussed in
the literature (e.g.~\cite{keremedis_weak_1999}). We formulate an auxiliary cardinal characteristic called
$\d^*$. Generalizing the proof of Lemma 6 to arbitrary $\p$ essentially amounts to proving
$\d^*\ge\b$,
the main result of the present paper.

In~\cite{gowersblog}, Gowers wonders why a purely combinatorial problem should have a solution with
non-trivial arguments involving proof theory and model theory. Indeed we establish here that the
proof path (c.f.~\cite{fridman-234}) of the original solution is suboptimal, in the sense that
the positive answer to Question~\ref{ques:rothberger} has a proof which only requires basic graduate
knowledge of set theory. Kunen's famous introductory text~\cite{kunen83} covers all the prerequisite
material.

The title might be considered misleading. For one thing, while the proof is completely
guided by considerations of measure, the only actual measure theory needed to understand the proof
of the main result are the fundamental properties of measure and the product measure construction,
which is again at the introductory graduate level (e.g.~\cite{Halmos}). For a fuller understanding,
ultraproduct measures are needed as well.

Notwithstanding, finding the correct measure was the main challenge in
solving Question~\ref{ques:rothberger}.
We were led to consider ``imaginary'' measures in the sense
that they only exist in forcing extensions of the mathematical universe, but what is
more they only exist in the most ``violent'' forcing extensions which collapse $\aleph_1$ to a
countable ordinal: Uncountable limit steps become countable in this forcing extension, putting us
back into Rothberger's original scenario. The need for this is explained in Section~\ref{sec:limit-as-products}.

As for Gowers' suggestion that perhaps a simple combinatorial argument was overlooked because of
false expectations that $\p<\t$ was consistent, we do not believe this to be the case. Our proof
establishes an apparently deep connection between cardinal characteristics of the reals and the nonstationary
ideal of ordinals of countable cofinality. To our knowledge no such connection has
ever been observed before.

Tsaban~\cite{tsaban_selection_2003} formulated the linear refinement number, now called $\lr$,
that captures the phenomenon of Rothberger's question.
Question~\ref{ques:rothberger} can be stated equivalently, and succinctly, as:

\begin{question}
  \label{ques:equiv}
  Does $\p<\b$ imply\/ $\p<\lr$?
\end{question}

Results on the cardinal characteristic $\lr$ shed further light on this
phenomenon in addition to our positive answer established here. Indeed, one of the results on $\lr$
in~\cite{machura_linear_2016} is seen to compliment our positive answer to Question~\ref{ques:equiv}
in Corollary~\ref{cor:ans-more}. While~\cite[Lemma~2.5]{machura_linear_2016} informs us that any
filter-base of cardinality smaller than $\lr$ can in fact be refined by a maximal tower of cardinality $\p$.
We believe the techniques developed here may be instrumental in answering open questions
about~$\lr$ (e.g.~Question~\ref{ques:tsaban}).

Hausdorff's discovery of $(\omega_1,\omega_1^*)$-gaps
and subsequent work on gaps was carried further by Rothberger,
leading up to his Question~\ref{ques:rothberger}, as discussed in~\S\ref{sec:gaps}.
It is then fitting that we find an application of our solution in the realm of gaps.
A result on the spectrum of \emph{universal gap} types in the ultrapowers $\irr\div\U$ is obtained in~\S\ref{sec:gaps}, closely related to the Malliaris--Shelah result on gaps in ultrapowers
$\irr\div\U$ for \emph{generic} ultrafilters~$\U$.

\subsection{Background}

A few years ago the author was deeply shocked when he stumbled across
Gowers' blog post~\cite{gowersblog}. The author's
first two works~\cite{hirschorn2000towers} and~\cite{hirschorn2000towersm}, done during and included in his
doctoral thesis~\cite{hirschorn2000cohen}, were on the topic of these two
cardinals. Up to the time the author was still following developments in set theory,
it seemed the unanimous consensus was that $\p<\t$ was consistent and it was
just a matter of devising the right forcing construction, with no one
taking the possibility $\p=\t$ serious enough to investigate. Indeed, prior to this result,
equalities between cardinal characteristics were without exception
relatively straightforward to prove and inevitability ``came out in the
wash'' when investigating their combinatorial properties.\footnote{Bell's
  Theorem~\cite{bell_combinatorial_1981}, used in the proof of the main result here, might be
  considered an exception. Though strictly speaking it is an equality between the cardinal
  characteristic $\p$ and a forcing axiom.}  

In~\cite{MalliarisShelah2016} the equality was established using model theoretic
methods. As pointed out by Gowers, this is hardly a satisfying
conclusion to the problem. This is a question of set theoretic
combinatorics with no obvious relation to model theory. Such a proof
obscures the true reason for the equality, in fact being inaccessible
to expert combinatorial set theorists who are unfamiliar with model
theory.

By way of analogy, a more extreme case would be a computerized proof
of a mathematical theorem, impenetrable by any human
mathematician. Such a proof could be of great value by
establishing the truth or falsehood of a mathematical
statement; however, one could hardly consider it an ideal resolution (this brings to mind the
famous 4-Color Problem, e.g.~\cite{tilley_will_2019}).

The author is unknowledgable in model theory, and made no attempt to
read the~\cite{MalliarisShelah2016} proof. Other set theorists
presumably had the same issue and numerous articles were published 
that apparently reproduce the~\cite{MalliarisShelah2016} proof while
attempting to minimize the amount of model theory required.

Besides solving Rothberger's more difficult problem,
we contribute a proof of $\p=\t$ which does not rely on any
model theory or proof theory, readily accessible to any set theorist.
As for the objection that these questions on cardinal characteristics are not about measure
theory either, we point out that connections with measure theory are made in the same
paper~\cite[Theorem~5b]{rothberger1948some} where the $\p=\t$ question is
introduced.

Our strategy to solve Question~\ref{ques:rothberger}
using measure theory is described in Section~\ref{sec:meas-theor-strat}.
In Section~\ref{sec:partition-measures}, we describe how a tower together with a partition of $\N$,
can be associated with a sequence of product measures, called partition measures, on~$\pN$.
In an earlier failed attempt, ultraproduct partition measures were used instead, and new nontrivial measure
theoretic results were used to obtain a limiting measure of this family (see
\S\ref{sec:ultraproducts})---giving rise to the title of the present work. However, without
collapsing cardinals such an attempt was doomed.  

Further progress on our measure theoretic strategy was only made after the author recalled his earlier
research~\cite{hirschorn_james_partition_2005}, done as a post-doctoral researcher in Japan, where
the fundamental mathematical concept of measure was generalized to a \emph{measure over some model
  $M$}. This approach is explored in Section~\ref{sec:limit-via-ultrafilter}, and the understanding gleaned,
of why it cannot be directly applied here,
led us straight to a suitable measure, used for the solution.

We mention that the work~\cite{hirschorn_james_partition_2005}
was motivated by the desire to put one of the main results from the author's earlier thesis
work  on random reals and Hausdorff gaps (\cite{hirschorn2000cohen}, \cite{hirschorn_summable_2003})
in a proper measure theoretic context, with the eventual
goal of discovering new combinatorial facts about the random forcing extension.
The author has produced a number of additional works on random reals and Hausdorff gaps
(\cite{hirschorn_random_2004}, \cite{hirschorn_strength_2008}, \cite{10.2307/40302758}).

During the course of the analysis in Section~\ref{sec:limit-via-ultrafilter}
we find that a special case of Rothberger's
Question~\ref{ques:rothberger} is already handled by the author's thesis work on random reals and towers in
$\pnfin$, i.e.~\emph{towers of measurable functions}, published in~\cite{hirschorn2000towersm}.

The main result of~\cite{hirschorn2000towersm} (among a number of other related results) is that the
operation of adding one random real does not cause the value of $\p$ to decrease, symbolically
$\|\dot \p\ge\check \p\|=1$. This left open the question of whether the cardinal $\p$ is in fact
invariant under adding one random real? Now that $\p=\t$ has been resolved, it is our opinion that
this invariance question is perhaps the most interesting unsolved problem concerning cardinal
characteristics, at least among those concerning the cardinal $\p$. See
Question~\ref{ques:optimality} for another candidate. 

\subsection{Preliminaries}

$\p$ and $\t$ are cardinal characteristics of the Boolean algebra $(\pnfin,\subseteqfnt)$.
This is the quotient of the field $\pN$ of all subsets of $\N$ over the equivalence relation of equality modulo finite, where $a$ is equivalent to $b$ modulo finite if their set difference $a\diff b$ is finite.
Thus $\pnfin$ is the Boolean algebra determined by the partial ordering of inclusion modulo finite,
where $[a]\subseteqfnt[b]$ iff $a\subseteqfnt b$ which means $a\setminus b$ is finite, for all $a,b\subseteq\N$.

$\p$~is the smallest cardinality of a (proper) filter-base on $\pnfin$
with no nonzero lower bound; in other words,
working with representatives $a\subseteq\N$ in place of $[a]\in\pnfin$,
a family $\F$ of infinite subsets of $\N$, all of whose finite intersections are infinite,
having no infinite \emph{pseudo-intersection}, i.e.~no infinite $x$
with $x\subseteqfnt a$ for all $a\in\F$. We call a $\supseteqfnt$-well-ordered subfamily of
$(\pnfin)^+$\footnote{$(B)^+$ denotes the nonzero elements of $B$.} a \emph{tower}. The cardinal $\t$ is the
smallest order-type of a maximal tower; in other words, working with representatives, a family 
$a_\xi\subseteq\N$ ($\xi<\theta$) for some ordinal $\theta$ with $a_\xi\supseteqfnt a_\eta$
 for all $\xi<\eta$ having no infinite pseudo-intersection.  
Since a well-ordered family in $(\pnfin)^+$ forms filter-base, $\p\le\t$.

That $\p$ is in fact uncountable is due to Hausdorff~\cite{hausdorff1936summen}.
(That work also includes a major discovery discussed in \S\ref{sec:gaps}.)

\begin{thm}[Hausdroff, 1936]
  $\p\ge\aleph_1$.
\end{thm}

In some partial order $(P,\le)$, $A\subseteq P$ is a \emph{refinement} of $B\subseteq P$ if every
$b\in\nobreak B$ has an $a\in A$ with $a\le b$.
The \emph{linear refinement number} $\lr$ is defined to be
the smallest cardinality of a filter-base with no $\supseteqfnt$-linearly ordered refinement.
It is directly relevant to the main question addressed here; for if $L\subseteq P$ is a linear
refinement of~$P$ and $W\subseteq L$ is a maximal well-ordered subset of $L$, then $W$ is also a
refinement. It follows immediately that  
Questions~\ref{ques:rothberger} and~\ref{ques:equiv} are equivalent.

$\p\le\lr$ is immediate from the definitions.
Moreover, it a result of Tsaban~\cite{tsaban_selection_2003} that $\lr\le\d$
(which is strengthened further in~\cite[Corollary~2.13]{machura_linear_2016}),
and in particular there always exist filter-bases which cannot be refined to towers.

\begin{thm}[Tsaban, 2003]
  $\lr\le\d$.
\end{thm}

$\b$ and $\d$ are cardinal characteristics of the lattice
$(\irrfin,\lefnt)$. The is the quotient of the lattice $\irr$ of all
of function from $\N$ into $\N$ over the equivalence relation of
equality modulo finite, i.e.~$[f]\sim[g]$ if $f(n)=g(n)$ for all but
finitely many $n$. This determines a lattice ordering $\lefnt$ where
$[f]\le[g]$ iff $f\lefnt g$ meaning that $f(n)\le g(n)$ for all but finitely
many $n$.

$\b$ is the smallest cardinality of a subset of
$\irrfin$ with no common $\lefnt$-upper bound; in other words, working
with representatives $f:\N\to\N$ in place of $[f]$, a subset
$\Hcal\subseteq\irr$ with no $h:\N\to\N$ \emph{eventually bounding}
every member of $\Hcal$, i.e.~no $h:\N\to\N$ such that $f\lefnt h$ for
all $f\in\Hcal$. The cardinal $\d$ is the smallest cardinality of a
cofinal subfamily $\D$ of $\irrfin$, i.e.~every $f\in\irr$ has a
$[d]\in\D$ with $f\lefnt d$. Since a cofinal family cannot be bounded, $\b\le\d$.

We have introduced a ranked list of cardinals here by the following
result\footnote{It is explicitly proved in~\cite{rothberger1948some}
  that $\b=\aleph_1$ implies $\t\le\aleph_1$. However the same proof
  generalizes to arbitrary cardinals.} (\cite[Theorem~3]{rothberger1948some}) mentioned in the introduction.

\begin{thm}[Rothberger, 1948]
  $\t\le\b$.
\end{thm}

The cardinal characteristics just introduced are called \emph{small
  cardinals} because they are bounded by the continuum
$\cntum=2^{\aleph_0}=|\pN|=|\reals|$ (e.g.~$\d\le\cntum$, trivially).

\subsection{Measure Theoretic Strategy}
\label{sec:meas-theor-strat}

In the limit step of Rothberger's proof of \cite[Lemma~7]{rothberger1948some}, one
must find ``thick'' pseudo-intersections (see~Section~\ref{sec:thick} for more details):
For a tower $|T|<\p$ and another family $H$ of nonorthogonal subsets of $\N$, one must find a pseudo-intersection of
the tower which is ``thick'' in the sense that it intersects every member of $H$. 

The strategy is based on the idea that this thickness can be measured; that is,
by an actual measure.
A distribution $(\pN,\M,\mu)$ is created on subsets of $\power(\N)$,
together with two transformations of subsets of $\N$ into measurable sets
$\varphi_T,\varphi_H:\pN\to\M$ with the property that
\begin{equation}
  \label{eq:17}
  \mu\bigl(\varphi_T(x)\cap\varphi_H(b)\bigr)>0\impls x\cap b\text{ is infinite.}
\end{equation}
Roughly, the idea is to first find a measure $\mu$ satisfying the property that $\mu\bigl(\varphi_H(b)\bigr)=1$ for all
$b\in H$, and then use the cardinality assumption on $T$ to produce a pseudo-intersection
$x\subseteqfnt T$ which is ``thick'', i.e.~$\mu\bigl(\varphi_T(x)\bigr)>0$. The problem of finding an extension of
$T$ of positive measure is a reduction of the original problem of trying to find such an extension
intersecting every member of $H$ where the cardinality of $H$ exceeds $\p$. 

Searching for such a measure one is quickly led to consider a partition of $\N$ into 
finite intervals $I_n$ such that $t\cap b\cap I_n\ne\emptyset$ for all but finitely many $n$, 
for~all~$t\in T$ and $b\in H$. 
The most obvious choice of measure is then a product $\mu$ of measures $\mu_n$ on $\power(I_n)$ for
each $n$;  
and the natural choice for these is to let $\mu_n$ be the Haar distribution on $\power(I_n)$,
i.e.~$\mu_n(\{s\subseteq I_n:s\ni k\})=\frac12$ for each $k\in I_n$.
The straightforward choices for $\varphi_T$ and $\varphi_H$ would be to let
$\varphi_T(x)=\{y:y\subseteqfnt x\}$  and $\varphi_H(b)=\{y:y\cap b\text{ is infinite}\}$.
This choice trivially satisfies~\eqref{eq:17}, and $\mu\bigl(\varphi_H(b)\bigr)=1$ for any infinite
$b$. However, for any nontrivial tower $T$ we will have $\varphi_T(t)$ $\mu$-null for essentially all $t\in T$. 

Instead, choosing the larger
$\varphi_T(x)=\{y:y\cap I_n\subseteq x$ for infinitely many~$n\}$, and then $\varphi_H(b)=\{y:y\cap b\text{
  intersects }I_n$ for all but finitely many $n\}$ so that~\eqref{eq:17} holds, we can hope to find
a suitable measure more easily. Then whenever
$b$ satisfies the summability condition
\begin{equation}
  \sum_{n=0}^\infty 2^{-|b\cap I_n|}<\infty,
\end{equation}
$\varphi_H(b)$ has Haar-measure one. Even though the Haar product measure is not suitable
to obtain non-null $\varphi_T(t)$, it leads down the right track.  

A more symmetrical alternative is considered in \S\ref{sec:ultraproduct-measure},
where $\varphi_T(t)=\{y:y\cap I_n\subseteq x$
for $\U$-many $n\}$ and $\varphi_H(b)=\{y:y$ intersects $b\cap I_n$ for $\U$-many $n\}$, and now $\mu$
is an ultraproduct measure $\lim_{n\to\U}\mu_n$ on the intervals.

\section{Thick Pseudo-Intersections}
\label{sec:thick}

Our objects of study are so called \emph{pseudo-intersection}s of
filter-bases of subsets of $\N$, which are ``thick'' in a specific
sense. This thickness is relative to a family $H$ of in infinite subsets of
$\N$, consisting of those subsets of $\N$ that have an infinite
intersection will all members of the family. 

Let $B$ be a Boolean algebra. For $G\subseteq B$, define $-G=\{-x:x\in G\}$.
Also define
\begin{equation}
\co(G)=\bigl(\downcl{(-G)}\bigr)^\complement.
\end{equation}
In other words, $\co(G)$ is the set of all members of $B$ which meet every
member of~$G$.
When $G$ generates a filter,
$\co(G)$ is the coideal of the dual ideal of the filter~$\upcl{G}$, and in
particular $G\subseteq\co(G)$. For an ultrafilter $U$, $\co(U)=U$. 

The following is well-known.

\begin{thm}
  \label{thm:d}
$\d$ is the smallest cardinality of a subset $[G]\subseteq\pN\div\Fin$
for which $\co([G])$ has a countable filter-base with no nonzero lower bound.
\end{thm}

\begin{rem}
  In other words,
$\d$ is the smallest cardinality of a subset $G\subseteq\pnfin$ for
which $\p\bigl(\co([G]),\subseteqfnt\bigr)=\aleph_0$.
\end{rem}

\begin{proof}
Suppose $G\subseteq\pN$ has cardinality $|G|<\d$, and that $\{[ a_k]:k<\omega\}$
is a filter-base in~$\co([G])$. For each $g\in G$, define $h_g:\N\to\N$ by
\begin{equation}
  h_g(n)=\min\left(g\cap\bigcap_{i=0}^n a_n\setminus n\right).
\end{equation}
Then $G$ is small enough that there must exist $f:\N\to\N$ such that
$f\nlefnt h_g$ for all~$g\in G$. Define
\begin{equation}
  a=\bigcup_{n=0}^\infty \bigcap_{k=0}^n a_k\cap\bigl[0,f(n)\bigr).
\end{equation}
Clearly $a\subseteqfnt a_n$ for all $n$, and $g\cap a$ is infinite for
all $g\in G$, because whenever $f(n)>h_g(n)$,
$g\cap\bigcap_{k=0}^na_k\cap[0,f(n)]\setminus n\ne\emptyset$.
Thus $[a]\in\co([G])$ is a nonzero lower bound of $\{[a_k]:k<\omega\}$.

Conversely, let $[F]\subseteq\irrfin$ be a cofinal subset of
cardinality $\d$.
For each $f\in F$, define $b_f\subseteq\N\times\N$ to be the area strictly above
the graph of $f$:
\begin{equation}
  \label{eq:22}
b_f=\{(i,n):n>f(i)\},
\end{equation}
and put $B=\{b_f:f\in F\}\subseteq\power(\N\times\N)$. For each $k$, define
$c_k\subseteq\N\times\N$ to
\begin{equation}
c_k =\{(i,n):i\ge k\}.
\end{equation}
Then $[C]=([c_k]:k<\omega)$ is clearly a decreasing sequence in $\co([B])$.
Supposing that $c\subseteqfnt c_k$ for all $k$, the vertical slices of $c$ are
all finite, i.e.~$\{i:(i, n)\in c\}$ is finite for all $n$, and we can define
$g\in\irrational$ by letting $g(n)$ be the maximum of the slice at $n$ (or
arbitrary if the slice is empty).
Now there exists $f\in F$ with $g\lefnt f$, which implies $c\cap b_f$ is finite,
proving the maximality of $[C]$ in $\co(B)$.
\end{proof}

Recall that the \emph{character} of a filter is the smallest cardinality of a filter-base generating
said filter.

\begin{definition}
Let $\d^*$ be the smallest cardinality of a family $[H]\subseteq \pN\div\Fin$
for which $\co([H])$ contains a filter of character less than $\p$ with no lower
bound in $\co([H])$.
\end{definition}

Working instead with representatives, $\d^*$ is the smallest cardinality of a family~$H$ of infinite
subsets of $\N$, for which there exists a filter-base $F$ on $\N$ of cardinality $|F|<\p$,
consisting entirely of \emph{thick} sets, i.e.~sets having an infinite intersection with every member of
$H$, but having no thick pseudo-intersection.

Note that without the restriction on the character of the filter, the
singleton $\{[\N]\}$ satisfies $\co(\{[\N]\})=\pN\div\Fin$, which has
a filter of character $\p$ with no lower bound by the very definition
of $\p$. 

\begin{prop}
  \label{prop:d-star-ub}
  $\d^*\le\d$. Moreover, $\p=\aleph_1$ implies $\d^*=\d$.
\end{prop}

\begin{lem}
  $\p\le\d^*$.
\end{lem}
\begin{proof}
  See for example~\cite[21A]{fremlin1984cons}.
\end{proof}

\begin{lem}
  \label{lem:refine-to-tower}
If $\p<\d^*$, then every filter-base of cardinality $\p$ can be refined to a tower.
\end{lem}

The proof of Lemma~\ref{lem:refine-to-tower} is the same idea as \cite[Lemma~6]{rothberger1948some}
but with $\d^*$ in place of $\b$.

\begin{proof}[Proof of Lemma~\ref{lem:refine-to-tower}]
  Let $x_\alpha\subseteq\N$ ($\alpha<\p$) enumerate representatives of a filter-base $[X]\subseteq\pnfin$.
  A refinement $t_\alpha\subseteq\N$ ($\alpha<\p$) is defined by recursion on $\alpha$ satisfying
  \begin{enumerate}
  \item\label{item:2} $t_\alpha\subseteqfnt t_\xi$ for all $\xi<\alpha$,
  \item $t_\alpha\subseteq x_\alpha$,
  \item\label{item:1} $t_\alpha\cap x_\xi$ is infinite for all $\xi<\p$.
  \end{enumerate}
  Condition~\eqref{item:1} states that $[t_\xi:\xi<\alpha]\subseteq\co([X])$. Therefore, as the
  $t_\xi$'s form a tower, and in particular a filter-base, and $|X|<\d^*$, there exists $t'_\alpha$
  satisfying conditions~\eqref{item:2} and~\eqref{item:1}. Now let $t_\alpha=t'_\alpha\cap x_\alpha$.
\end{proof}

Notice this gives the following strengthening of the result of Rothberger \cite[Lemma
6]{rothberger1948some}. 

\begin{cor}
  \label{cor:strengthen}
  If $\p=\aleph_1$ then $\d>\p$ implies that every filter-base of cardinality $\p$ can be refined to
  a tower.
\end{cor}

The previous two results can be expressed in terms of $\lr$, as can Question~\ref{ques:optimality}
below. Lemma~\ref{lem:refine-to-tower} can be
stated equivalently as $\p<\d^*$ implies $\p<\lr$, or in the contrapositive as $\p=\lr$ implies
$\p=\d^*$. Corollary~\ref{cor:strengthen} can be
expressed as $\lr=\aleph_1$ implies $\d=\aleph_1$,
which appeared earlier as~\cite[Theorem~2.2]{machura_linear_2016}. 

\begin{cor}
  If $\t\le\d^*$ then $\p=\t$.
\end{cor}

The main result of this paper the following.

\begin{thm}
  \label{thm:main}
  $\b\le\d^*$
\end{thm}

This answers Rothberger's Question~\ref{ques:rothberger} affirmatively.

\begin{cor}
  If $\p<\b$, then every filter-base of cardinality $\p$ can be refined to a tower. 
\end{cor}

Even more can be said by expressing this using the cardinal characteristic $\lr$.

\begin{cor}
  \label{cor:ans-more}
  $\p<\b$ implies that\/ $\p<\lr\le\b$.
\end{cor}
\begin{proof}
  Using~\cite[Theorem~2.7]{machura_linear_2016}.
\end{proof}

Left unanswered are the optimal bounds on $\d^*$. For example, playing the same game as with
Rothberger's questions one can ask whether the cardinality assumption on $\p$ in
Proposition~\ref{prop:d-star-ub} can be dropped:

\begin{question}
  \label{ques:equal}
  Does $\d^*=\d$?
\end{question}

This in turn is a question about the optimality of our strengthening of Rothberger's result:

\begin{question}
  \label{ques:optimality}
  If $\p<\d$, can every filter-base of cardinality $\p$ be refined to a tower?
\end{question}

The following question posed by Tsaban (private communication) is a promising line of inquiry in this
direction. A positive answer would be a major development as a strengthening
of our main result, and thus an even stronger version of $\p=\t$; while a negative answer would
provide important information on bounds for $\d^*$ by refuting the equality of Question~\ref{ques:equal}.

\begin{question}[Tsaban]
  \label{ques:tsaban}
  Is\/ $\lr\le\d^*$?
\end{question}

\section{Ultraproducts}
\label{sec:ultraproducts}

For purposes of this section, we let $I$ be an arbitrary index set.
An \emph{ultrafilter on $I$} is a maximal proper filter $\U\subseteq\power(I)$.
It thus has the property that for any $x\subseteq I$,
exactly one of $x$ or its complement is an element of $\U$.

The \emph{set ultraproduct} of a family of sets $A_i$ ($i\in I$), denoted
$\prod_{i\in I}A_i\div\U$, or simply $[A]$ when clear from the context, is the Cartesian product $\prod_{i\in I}A_i$ modulo
the equivalence relation $x\equiv y\pmod\U$ if $\{i\in I:x(i)=y(i)\}\in\U$.
For a sequence $\F_i$ of fields of subsets of $X_i$ ($i\in I$),
the ultrafilter property entails that
the Boolean algebra operations on
the field\footnote{Here $\prod_iA_i\div\U$ consists of equivalence classes over the product
  $\prod_iX_i$.}\label{def-ultra-alg} 
$\F_\U=\bigl\{\prod_{i\in I}A_i\div\U:A_i\in\F_i\bigr\}$ correspond to the coordinate-wise
operations; for example, $[A]\cup[B]=\prod_{i\in I}A_i\cup\nobreak B_i\div\U$,
and similarly for intersections $\cap$ and complements $-$.   

Suppose that
$(\vomega_i,\M_i,\mu_i)$ ($i\in I$) is a sequence of measure spaces and $\U$
is an ultrafilter on $I$. Then the \emph{measure ultraproduct} of the
sequence with respect to $\U$ is the triple $(\vomega_\U,\sigma(\M_\U),
\mu_\U)$, where $\vomega_\U=\prod_{i\in I}\vomega_i\div \U$ is the set
ultraproduct, $\sigma(\M_\U)$ is the $\sigma$-algebra generated by $\M_\U$ and
\begin{equation}
  \mu_\U([A])=\lim_{i\to\U}\mu_i(A_i)
\end{equation}
for each $[A]\in\M_\U$ is well-defined and finitely additive. So long
as the $\mu_i(\vomega_i)$ are mutually bounded, $\mu_\U$ will take nonnegative real
values; otherwise, the limit is taken in the compact space
$[0,\infty)\cup\{\infty\}$.

With the following result the construction of $\mu_\U$ extends to a
measure on the $\sigma$-algebra generated by $\M_\U$. The ultraproduct
measure construction is due to Loeb~\cite{Loeb1975} in the setting of
nonstandard analysis. 

\begin{lem}[Loeb]
  \label{lem:loeb}
  If the ultrafilter $\U$ is not $\sigma$-complete, then the union of
  any countable infinite pairwise disjoint subset of $\M_\U$ is not a
  member of $\M_\U$. 
\end{lem}
\begin{proof}
  For a sequence $A_0,A_1,\dots$ of elements of $\M_\U$,
given $A\in\M_\U$ with $A\subseteq\bigcup_{n=0}^\infty A_n$, it suffices to show
that $A\subseteq
\bigcup_{k=0}^nA_k$ for some $n$. Supposing to the contrary, let $[B]=A$ and
recursively choose
  representatives $[B_k]=\bigcup_{i=0}^k A_i$, $U_k\in\U$ and $x_{k}\in B$
  such that
  \begin{enumerate}
  \item  $U_{k+1}\subseteq U_k$,
  \item $x_k(i)\notin B_k(i)$ for all $i\in U_k$,
  \item $x_{k+1}(i)=x_k(i)$ for all $i\notin U_{k+1}$,
  \end{enumerate}
  Since $\U$ is countably incomplete, we may assume that
  $\bigcap_{k=0}^\infty U_k=\emptyset$, and thus
  define $x\in\prod_{i\in I}\vomega_i$ by
  $x(i)=\lim_{k\to\infty}x_{k}(i)$.
  Now for every $k$, $x(i)\notin B_k$ for all~$i\in U_k$, so that
  $[x]\notin A_k$,
  yet each $x(i)\in B(i)$ so that $[x]\in A$, a contradiction.
\end{proof}

Now it follows from the Carath\'eodory Extension Theorem that $\mu_\U$
extends to a measure on $\sigma(\M_\U)$, where the extension is
moreover unique in the bounded case.

Topological measure spaces do not play a role in the results presented here.
However, they did in a prior attempt, where the set ultraproduct of a sequence of topological spaces
has the  naturally occurring topology generated by
the sets $\prod_{i\in I}V_i\div\U$ where each $V_i\subseteq X_i$ is open.
If each of the factors $X_i$ is
zero-dimensional, this is preserved under the ultraproduct and the
Boolean algebra $\bigl\{\prod_{i\in I}B_i\div\U:B_i\in\clopen(X_i)\bigr\}$ is a basis
of clopen sets for the topology on the ultraproduct of the $X_i$'s.

It was proven that Radon measures are preserved under ultraproducts, to our knowledge a new
result. With Radon measures one can apply fundamental results in functional analysis such as the
Riesz--Markov--Kakutani Representation Theorem, which in turn can be used to take a topological
limit of a sequence of Radon measures. In our case the ultraproduct Radon measures associated with the
partition measures (\S\ref{sec:partition-measures}) of a tower.   

The proof that Radon measures are preserved employs the \emph{Banach space ultraproduct} of Banach
spaces of the from $C_0(X_i)$~(see e.g.~\cite{heinrich1980ultraproducts}), where the equivalence classes
of the product are modulo the linear subspace of all $f\in\nobreak\ell_\infty\bigl(\prod_{i}C_0(X_i)\bigr)$
with $\lim_{i\to\U}\|f_i\|=0$. 

\subsection{Limits of Measures via Ultrapowers}

The following is \cite[Definition~3.1]{hirschorn_james_partition_2005}. An ultrafilter over the entire universe $V$
coincides with the usual notion of an ultrafilter. 

\begin{definition}
  For a class M and an index set $I \in M$, a filter $\V$ on I is
an \emph{ultrafilter over $M$} if either $X \in F$ or $I \setminus X \in \V$, for all $X \subseteq I$ in $M$.
A filter $\V$ on $I$ is \emph{$\kappa$-complete over $M$} if for every subset $\F\subseteq\V$ in $M$
with $M\models|\F|<\kappa$, $\bigcap\F\in\V$ ($\sigma$-complete means $\aleph_1^M$-complete,
c.f.~Definition~\ref{def:bool-alg-over}).
\end{definition}

In what follows $M$ is assumed to be a transitive model satisfying enough of $\zfc$ to ensure the
validity of the following straightforwards facts (e.g.~Lemma~\ref{lem:ultra-limit-M}).

The terminology of an \emph{absolutely compact and
  Hausdorff} topology means that it is compact Hausdorff and has a definition which is sufficiently
absolute between transitive models.
We shall not go into the details, but simply note that one can take $K$ to be a closed interval
or $[0,\infty)\cup\{\infty\}$, and easily prove the following known fact for in these cases
(c.f.~\cite{hirschorn_james_partition_2005}).

\begin{lem}
  \label{lem:ultra-limit-M}
  Let $\V$ be an ultrafilter on $I$ over $M$. If\/~$f:I\to K$ is a
  function in~$M$, where $K$ is absolutely compact and Hausdorff,
  then the limit $\lim_{i\to\V}f(i)$ exists and is unique.
\end{lem}

Let $\V$ be an ultrafilter on $I$ over a model $M$.
For sequence in $M$ of fields $\F_i$ on $X_i$ ($i\in I$),
the construction on page~\pageref{def-ultra-alg} generalizes to
\begin{equation}
  \F_\V^M=\biggl\{\prod_{i\in I}A_i\div\V:(A_i:i\in I)\in M\biggr\}
\end{equation}
which is also a Boolean algebra whose
operations correspond to the coordinate-wise operations.

The following definitions from \cite[Definition 2.1]{hirschorn_james_partition_2005}
generalize the fundamental mathematical notion of a measure.

\begin{definition}
  \label{def:bool-alg-over}
  A Boolean algebra $\B$ of subsets of $X$ is a \emph{$\sigma$-algebra
    over $M$} if $\bigcup_{n=0}^\infty E_n\in \B$ for every sequence
  $(E_n:n\in\N)\in M$ of elements of~$\B$.
\end{definition}

\begin{definition}
  \label{def:meas-over}
Let $M$ be a class containing a set $X\in M$, and let $\F\subseteq M$
be a field of subsets of $X$. A function $\mu:\F\to[0,\infty]$ is a \emph{measure on $X$ over $M$} if
\begin{myenumerate}
\item $\mu(\emptyset) = 0$,
\item for every sequence $(E_n: n \in N)$ of pairwise disjoint
members of $\F$, where
$(E_n : n \in\N) \in M$, if $\bigcup_{n=0}^\infty E_n \in \F$ then
$\mu\bigl(\bigcup_{n=0}^\infty E_n\bigr)=\sum_{n=0}^\infty \mu(E_n )$.
\end{myenumerate}
And a \emph{measure space over $M$} is a triple $(X, \M, \mu)$ where $\M$ is a $\sigma$-algebra of
subsets of $X$ over $M$ and $\mu:\M\to[0,\infty]$ is a measure on $X$
over $M$.
\end{definition}

\begin{lem}
  \label{lem:sigma-over-M}
  Suppose that $\M_i$, $X_i$ \tu($i\in I$\tu) are sequences, in $M$,
  of $\sigma$-algebras $\M_i$ of subsets of $X_i$,
  and suppose $\V$ is $\sigma$-complete ultrafilter over $M$.
  If $([E_n]:n\in\N)\in M$ is a sequence of members of $\M_\V^M$, then
  \begin{equation}
    \label{eq:sigma-union}
    \bigcup_{n=0}^\infty[E_n]=[E],
  \end{equation}
  where $E(i)=\bigcup_{n=0}^\infty E_n(i)$ for each $i\in I$. 
\end{lem}
\begin{proof}
  Suppose that $([E_n]:n\in\N)\in M$ is a sequence of members of $\M_\V^M$,
  and let $E$ be defined as above.
  Clearly $[E_n]\subseteq[E]$ for every $n$.
  Conversely, take $x\in E$.
  Then letting $J_n=\{i\in\nolinebreak I:x(i)\in E_n(i)\}$ ($n\in\N$),
  $I=\bigcup_{n=0}^\infty J_n$.
  Since the sequence $(J_n:n\in\nolinebreak\N)$ is in $M$, by $\sigma$-completeness over $M$,
  $J_n\in\V$ for some $n\in\N$. Thus $[x]\in [E_n]$.
\end{proof}

The following results for $\sigma$-complete ultrafilters over $M$ are in contrast to Loeb's result
(Lemma~\ref{lem:loeb}) for non-$\sigma$-complete ultrafilters. 

\begin{cor}
  \label{cor:sigma-over-M}
  Suppose that $\M_i$, $X_i$ \tu($i\in I$\tu) are sequences, in $M$,
  of $\sigma$-algebras $\M_i$ of subsets of $X_i$,
  and suppose $\V$ is $\sigma$-complete ultrafilter over $M$.
  Then $\M_\V^M$ is a $\sigma$-algebra of subsets of\/ $\prod_{i\in I}X_i\div\V$ over $M$. 
\end{cor}

\begin{lem}
  \label{lem:lim-measure-M}
  Suppose that $(\vomega_i,\M_i,\mu_i)$ \tu($i\in I$\tu) is a sequence of
  measures spaces, in $M$, and that $\V$ is a $\sigma$-complete ultrafilter on $I$
  over~$M$. By Lemma~\tu{\ref{lem:ultra-limit-M}},
  \begin{equation}
    \label{eq:limit-over-M}
    \mu_\V([A])=\lim_{i\to\V}\mu_i(A_i)
  \end{equation}
  defines a function from $\M_\V^M$ into $[0,\infty]$. 
  Then $\mu_\V$ is measure on $\prod_{i\in I}\vomega_i\div\V$ over~$M$. If the $\mu_i$'s are probability measures then so is $\mu_\V$.
\end{lem}
\begin{proof}
  As $\mu_\V$ is finitely additive, countably additivity needs to be established.
  We apply \cite[Lemma~2.2]{hirschorn_james_partition_2005} which extends a basic fact of measure
  theory to measures over a model: 
  Supposing that $[E_0]\supseteq [E_1]\supseteq\cdots$ is a sequence in $M$ whose intersection is empty,
  it is required to prove that $\inf_{n\in\N}\mu_\V([E_n])=0$.

  Let $J$ be the set of all $i\in I$ for which $\bigcap_{n=0}^\infty E_n(i)=0$. Then $J\in\V$ as
  otherwise $J\in M$ would imply that its complement is in $\V$ and the intersection of the $[E_n]$ would be
  nonempty. Similarly, $J_{m,n}=\{i:E_m(i)\subseteq E_n(i)\}\in\V$ for all $m<n$. Since the sequence
  is in $M$, $J'=J\cap\bigcap_{m=0}^\infty\bigcap_{n=m+1}^\infty J_{m,n}\in\V$ by
  $\sigma$-completeness.

  Suppose towards a contradiction that $\varepsilon=\inf_{n}\mu_\V([E_n])>0$, and pick a rational
  $0<q\le\varepsilon$. Then as $K_n=\bigl\{i:\mu_i\bigl(E_n(i)\bigr)<q\bigr\}\in M$, $K_n\notin \V$
  for all $n$. Since the whole sequence $(K_n:n\in\N)$ is in $M$, there exists $i_0\in
  J'\setminus\bigcup_{n}K_n$. But $E_0(i_0)\supseteq E_1(i_0)\supseteq\cdots$ and has an empty
  intersection and thus  $\inf_n\mu_{i_0}\bigl(E_n(i_0)\bigr)=0$, a contradiction because then
  $i_0\in K_n$ for some $n$. 
\end{proof}

Now we turn to the special case where all of the $X_i$'s and $\M_i$'s are the
same, say $X$ and a $\sigma$-algebra $\M$ of subsets of $X$, but the measures $\mu_i$ vary over
$I$. 
Define $\wh x\in X^I$ by
\begin{equation}
  \wh x(i)=x\espc\text{for all $i\in I$},
\end{equation}
and for each $A\in\M$, define $\wh A\in\M^I$ by
\begin{equation}
  \wh A(i)=A\espc\text{for all $i\in I$}.
\end{equation}
Then $\mu_\V$ defines an ultrapower measure on $X^I$,
and we can define the \emph{ultrapower limit}\/ $\nu_\V$ of the $\mu_i$'s on $\M$ by
\begin{equation}
  \nu_\V(A)=\mu_\V([\wh A])=\lim_{i\to\V}\mu_i(A)\espc\text{for all $A\in \M$}.
\end{equation}

\begin{lem}
  \label{lem:meas-over}
  If $\V$ is a $\sigma$-complete ultrafilter over $M$, then
  $\nu_\V$ is a measure on $X$ over~$M$.
\end{lem}
\begin{proof}
  Take a sequence $(E_n:n\in\N)$ in $M$ of members of $\M$.
  Put $E=\bigcup_{n=0}^\infty E_n$. Then
  \begin{equation}
    \bigcup_{n=0}^\infty [\wh E_n]=[\wh E]
  \end{equation}
  by Lemma~\ref{lem:sigma-over-M}. Thus $\nu_\V(E)=\mu_\V([\wh
  E])=\sum_{n=0}^\infty \mu_\V([\wh E_n])=\sum_{n=0}^\infty
  \nu_\V(E_n)$, as $\mu_\V$ is a measure over $M$.
\end{proof}

\subsection{Forcing a measure over a model}

Forcing is the natural method for producing ultrafilters and measures over a model.
Let $\ideal_\theta$ denote the ideal of bounded subsets of the ordinal $\theta$.

\begin{lem}
  \label{lem:force-ultra-over}
The generic ultrafilter $\V$ of the forcing notion $\power(\theta)\div\ideal_\theta$ is a $\cof(\theta)$-complete
ultrafilter on $\theta$ over the ground model, consisting of cofinal subsets of\/~$\theta$. 
\end{lem}
\begin{proof}
  Take $X\in\ideal_\theta^+$ (in the ground model) and $Y\subseteq\theta$ in
  the ground model. If $X\cap Y$ is cofinal, i.e.~$X\cap Y\in\ideal_\theta^+$,
  then it forces $Y\in\V$; otherwise, $X\setminus Y$ must be cofinal and
  forces $\theta\setminus Y\in\V$.

  For $\cof(\theta)$-completeness, if $X$ forces that $\Y\subseteq\V$ where $\Y$ is a family in the
  ground model of cardinality less than $\cof(\theta)$,
  then $X\setminus\bigcap\Y=\bigcup_{Y\in\Y}X\setminus Y$ is bounded, and thus 
  $X$ forces that $\bigcap\Y\in\V$.
\end{proof}

\begin{rem}
  Absoluteness is not an issue for our present consideration of measure. In our usage of a measure
  $\mu$ over $M$ (cf.~Definition~\ref{def:meas-over}) objects in $M$ are being measured. Suppose, for
  example, $\F$ is the algebra of finite unions of half-open subintervals of~$[0,1]$, and $\mu\in N$
  is a measure on $[0,1]$ over $M\subseteq N$ with domain including $\F$. Then $\mu$ is measuring $[a,b)^M$ and not $[a,b)^N$
  (in general $[a,b)^M\subsetneq[a,b)^N$).  

  If we wanted to instead to treat the domain of the measure as consisting of definable sets
  (e.g. Borel sets), as is typical when relativizing between models,
  for example when forcing (see e.g.~\cite{solovay1970model}), so that
  $\mu$ measures the interval defined by $[a,b)$, then Definition~\ref{def:meas-over} could
  be modified accordingly. However, there was no need for absoluteness in the present work,
  especially as it does not use topological measure theory, nor was there any need
  in~\cite{hirschorn_james_partition_2005}.
\end{rem}

\section{Measures Induced by Towers}
\label{sec:meas-induc-tower}

For a fixed partition on $\N$, every subset of $\N$ has an associated \emph{partition measure} on
$\pN$ as described in Section~\ref{sec:partition-measures}. The
measure induced by a tower $[T]$ is a limit of the partition measures 
associated with its representatives. The crux of the matter is how to
take the limit of these measures. Earlier we alluded to a prior attempt to take the topological
limit of the Radon ultraproduct measures derived from the partition measures. In
Sections~\ref{sec:limit-via-ultrafilter} and~\ref{sec:limit-as-products},
two different limiting approaches are considered, respectively. 

Throughout this section, $[T]\subseteq\pnfin$ is a tower of order-type
$\theta<\p$, with representatives $(t_\alpha:\alpha<\theta)$ enumerating
$T$ according to its well-ordering. 

\subsection{Partition Measures}
\label{sec:partition-measures}

\begin{definition}
  $\pi:\N\to\N$ is called a \emph{partition function \tu(of $\N$\tu)}
  if $\pi(0)=0$ and $\pi$ is strictly increasing, because the intervals
  \begin{equation}
    I_\pi(n)=\bigl[\pi(n),\pi(n+1)\bigr)
  \end{equation}
  form a partition of $\N$.
\end{definition}

\begin{lem}
  \label{lem:summable-partition}
  Suppose $H$ is a family of infinite subsets of\/ $\N$ with cardinality
  $|H|<\b$, and $f\in\irr$ is a given function.
  Then there exists a partition function $\pi:\N\to\N$ such that
  \begin{equation}
    \label{eq:1}
    |b\cap I_\pi(n)|\ge f(n)\espc\textup{for all but finitely many $n$},
  \end{equation}
  for all $b\in H$.
\end{lem}
\begin{proof}
  For any infinite $b\subseteq\N$ let $e_b:\N\to\N$ be its enumerating
  function. For each $b\in H$, let $g_b\in\irr$ be given by
  \begin{equation}
    g_b(n)=e_b\bigl(f(n) + n\bigr).
  \end{equation}
  By cardinality considerations, there exists a $g\in\irr$ 
  satisfying $g_b\lefnt g$ for all $b\in H$, and also $g(n)>n$ for
  all $n$. Assuming without loss of generality that $f$ is increasing,
  $\pi(0)=0$ and $\pi(n+1)=g\bigl(\pi(n)\bigr)$
  defines a partition function satisfying~\eqref{eq:1}.
\end{proof}

Let each $\C_\pi(n)$ be the family of all subsets
of~$I_\pi(n)$ endowed with the discrete topology, so that 
\begin{equation}
\C_\pi=\prod_{n=0}^\infty\power(I_\pi(n))
\end{equation}
is a Cantor space with its product topology. $\C_\pi$ identifies homeomorphically with $\pN$
via the natural identification $x\mapsto\bigcup_{n=0}^\infty x(n)$,
where $\pN$ has a topology via its identification with the Cantor set
$\{0,1\}^\N$.  If $\tilde x\in\C_\pi$ is the identification of $x\subseteq\N$ then
\begin{equation}
  \tilde x(n)=x\cap I_\pi(n)\espc\text{for all $n$.}
\end{equation}
For $U\subseteq \power(I_\pi(n))$ the notation used for the corresponding clopen set is
\begin{equation}
  \<n,U\>_\pi=\<U\>_\pi=\<U\>=\{x\in\C_\pi:x(n)\in U\}.
\end{equation}
We can omit the index $n$ because it can be inferred from $U$ as $\pi$ is a partition.

For a nonempty finite set $S$, we let $\nu_S$ denote the standard probability
measure on~$\power(S)$: 
Let $N=|S|$. Then $\power(S)\cong\{0,1\}^N$, and the Haar product
measure $\mu$ on~$\{0,1\}^N$ is determined by $\mu(\{x:x(k)=\ell\})=\frac12$
for all $k=0,\dots,N-\nolinebreak1$ and all~$\ell=0,1$. Thus the basic clopen
subsets of $\power(S)$ are of the form $\{x\subseteq\nolinebreak S:\nolinebreak s\in\nolinebreak x\}$ and
$\{x\subseteq S:s\notin x\}$ ($s\in S$) and each is of measure $\frac12$.

More generally, for any set $a$ intersecting $S$, we let
\begin{equation}
  \nu_{S,a}(X)=\nu_{S\cap a}\bigl(X\cap \power(a)\bigr)
\end{equation}
be the probability measure on $\power(S)$ induced by $\nu_{S\cap a}$. 

Each $a\subseteq\N$ determines a sequence of measures $\mu_{a,n;\pi}$ on
$\power(I_\pi(n))$ ($n\in\N$) by
\begin{equation}
  \label{eq:4}
  \mu_{a,n;\pi}(A_n)=\nu_{I_\pi(n),a}(A_n)=\nu_{\tilde a(n)}(A_n)
  \espc\text{for all $A_n\subseteq \power(I_\pi(n))$}.
\end{equation}
The \emph{partition measure} on $\C_\pi$ associated with $a$ is the product measure
\begin{equation}
  \mu_a=\mu_{a;\pi}=\prod_{n=0}^\infty\mu_{a,n;\pi}.
\end{equation}

\begin{prop}
  \label{prop:prod-meas}
  $\mu_a(\<U\>)=\nu_{I_\pi(n),a}(U)$ for all\/ $U\subseteq\power(I_\pi(n))$.
\end{prop}

Recall that a measure $\mu$ is called \emph{nonatomic} if $\mu(\{x\})=0$ for every singleton. So
long as $\pi$ partitions $a\subseteq\N$ into pieces of unbounded size,
the product measure associated with some $a$ is nonatomic. 

\begin{prop}
  \label{sec:meas-nonatomic}
  Let $a\in\C_\pi$.
  If $\limsup_{n\to\infty}|a(n)|=\infty$ then $\mu_{a;\pi}$ is nonatomic.
\end{prop}

For $z\in\C_\pi$ define a $G_\delta$-set by
\begin{equation}
  \pinf(z)=\bigcap_{k=0}^\infty\bigcup_{n=k}^\infty\<\power(z(n))\>,
\end{equation}
so that $x\in \pinf(z)$ iff $\{n:x(n)\subseteq z(n)\}$ is
infinite. For $z\subseteq\N$, we can write $\pinf[\pi](z)=\pinf(z)$ to
emphasis that $z$ is being identified with an element of $\C_\pi$.

\begin{prop}
  \label{prop:pinf-inc}
  $\pinf(y)\subseteq\pinf(z)$ whenever $y\subseteqfnt z$.
\end{prop}

\begin{prop}
  \label{prop:decr-meas}
  $\mu_{t;\pi}\bigl(\pinf[\pi](z)\bigr)\le \mu_{v;\pi}\bigl(\pinf[\pi](z)\bigr)$
  whenever $z\subseteqfnt v\subseteqfnt t$.
\end{prop}
\begin{proof}
  For all but finitely many $n$, $z(n)\subseteq v(n)\subseteq t(n)$ and thus
  $\nu_{I_\pi(n),v}(\power(z(n)))=2^{|z(n)|-|v(n)|}\ge2^{|z(n)|-|t(n)|}= \nu_{I_\pi(n),t}(\power(z(n)))$.
\end{proof}

Consider the smaller $F_\delta$-set $\power^*(x)=\{y:y\subseteqfnt x\}\subseteq\pinf(x)$, which does
not depend on $\pi$.

\begin{prop}
  \label{prop:pow-star}
  $\mu_{t;\pi}\bigl(\power^*(u)\bigr)=1$ whenever $t\subseteqfnt u$.
\end{prop}

For a set $S$,
let $\intersects{S}(a)=\{x\subseteq S:x\cap a\ne\emptyset\}=\power(S)\setminus\power(S\setminus a)$.
And for~$z\in\C_\pi$, define
\begin{equation}
  \intfin(z)=\bigcup_{k=0}^\infty\bigcap_{n=k}^\infty\bigl\<\intersects{I_\pi(n)}\bigl(z(n)\bigr)\bigr\>,
\end{equation}
so that $x\in\intfin(z)$ iff $x(n)\cap z(n)\ne\emptyset$ for all but
finitely many $n$, and write
\begin{equation}
  \intfin[\pi](z)=\intfin(\tilde z)\espc\text{for $z\subseteq\N$,}
\end{equation}
to identify $z$ with $\tilde z\in\C_\pi$.
Notice that
\begin{equation}
\pinf(z)=-\intfin(-z).
\end{equation}

\begin{prop}
  \label{prop:p-int-h}
  For every\/ $x,y\in\C_\pi$, if $\pinf(x)\cap\intfin(y)\ne\emptyset$
  then $x(n)$ intersects $y(n)$ for infinitely many $n$.
\end{prop}

\begin{lem}
  \label{lem:meas-H}
  Suppose that $a,b\subseteq S$. Then
  \begin{equation}
    \nu_{S,a}\bigl(\intersects{S}(b)\bigr)=1-2^{-|a\cap b|}.
  \end{equation}
\end{lem}
\begin{proof}
  We have the following conditional probability in the Haar measure:\linebreak
  $P[x\cap\nolinebreak b=\nolinebreak\emptyset \cond x\subseteq a]=2^{-|a\cap b|}$.
  Since $\intersects{S}(b)\cap\power(a)=\{x\subseteq a:x\cap b\ne\emptyset\}$,  
  $\nu_{S,a}\bigl(\intersects{S}(b)\bigr)
  =\nu_{a}(\{x\subseteq a:x\cap b\ne\emptyset\})=1-2^{-|a\cap b|}$.
\end{proof}

The following property is related to the notion of a \emph{summable gap}
(\cite{hirschorn_summable_2003}).

\begin{cor}
  \label{cor:meas-one-h}
  Let $t,b\in\C_\pi$. If
  \begin{equation}
    \label{eq:11}
    \sum_{n=0}^\infty 2^{-|t(n)\cap b(n)|}<\infty
  \end{equation}
 then $\mu_{t;\pi}\bigl(\intfin(b)\bigl)=1$. Otherwise, $\mu_{t;\pi}\bigl(\intfin(b)\bigr)=0$.
\end{cor}
\begin{proof}
  Since $\intfin(b)=\bigcup_{k=0}^\infty
  \bigcap_{n=k}^\infty\bigl\<\intersects{I_\pi(n)}\bigl(b(n)\bigr)\bigr\>$, by Lemma~\ref{lem:meas-H}
  its measure is the supremum of the tails of the product $\prod_{n=0}^\infty1-2^{-|t(n)\cap
    b(n)|}$, which by  the Cauchy criterion for products converges iff the summability
  condition~\eqref{eq:11} holds.
\end{proof}

\subsection{Limit via an Ultrapower}
\label{sec:limit-via-ultrafilter}

Next we describe a limiting procedure to produce a measure over the
ground model following~\cite{hirschorn_james_partition_2005}.
There a measure over the ground model is constructed from the generic ultrafilter of
$\power(\omega_1)\div\ideal_{\omega_1}=\power(\omega_1)\div\mathrm{Count}$ (countable);
while here we more generally consider the ultrapower limit measures for generic ultrafilters on
$\power(\theta)\div\ideal_\theta$ for $\theta$ of uncountable cofinality.
Though this is not necessary for the
proof of the main Theorem~\ref{thm:main}, it is required for additional insights and
describes the gateway to the actual proof.

To simplify notation, write
\begin{equation}
  \mu_\alpha=\mu_{t_\alpha;\pi}
\end{equation}
for the partition measure for $t_\alpha$. For a $\power(\theta)\div\ideal_\theta$-generic $\V$ over
the ground model, the ultrapower limit
\begin{equation}
  \nu_T=\lim_{\alpha\to\V}\mu_\alpha
\end{equation}
defines a $\cof(\theta)$-additive probability measure on $\C_\pi$ over the ground
model by Lemma~\ref{lem:meas-over}.

This measure (over the ground model) satisfies the two basic ingredients to be plausible for our
measure theoretic strategy using the transformations $x\mapsto\pinf(x)$ and $x\mapsto\intfin(x)$.
The next Lemma in particular shows that $\pinf[\pi](t_\alpha)$ has measure one
for all $\alpha$. 

\begin{lem}
  $\nu_T\bigl(\power^*(t_\xi)\bigr)=1$ for all\/ $\alpha<\theta$. 
\end{lem}
\begin{proof}
  By Proposition~\ref{prop:pow-star},
  $\mu_\alpha\bigl(\power^*(t_\xi)\bigr)=1$ for all $\alpha\ge\xi$, and $\theta\setminus\xi\in\V$.
\end{proof}

\begin{prop}
  Whenever $b\in\C_\pi$ satisfies the summability condition~\eqref{eq:11} for all $t\in T$,
  $\nu_T\bigl(\intfin(b)\bigr)=1$.
\end{prop}
\begin{proof}
  By Corollary~\ref{cor:meas-one-h}.
\end{proof}

Indeed, assuming $|\Hcal|<\b$, using Lemma~\ref{lem:summable-partition} we can find a partition $\pi$
so that
\begin{equation}
  \label{eq:19}
  |t_\alpha\cap b\cap I_\pi(n)|\ge n\espc\text{for all but finitely many $n$,}
\end{equation}
for all $b\in H$, which in particular implies the summability condition~\eqref{eq:11} for all
$t_\alpha\in T$ and $b\in H$.

The ultrapower measure is classified into one of two mutually exclusive cases.
\begin{equation}
  \lim_{\alpha\to\theta}\mu_\xi\bigl(\pinf(t_\alpha)\bigr)=0\espc\text{for all $\xi<\theta$.}
  \tag{Essential Case}
\end{equation}

We now demonstrate that the nonessential case is handled by the theory developed
in~\cite{hirschorn2000towersm}.  
In this latter case there exists $\alpha_0$ such that
$\mu_{\alpha_0}\bigl(\pinf(t_\alpha)\bigr)>0$ for cofinally many $\alpha$. 
Then letting $\R_{\alpha_0}$ be the measure algebra of the 
measure space $(\C_\pi,\mu_{\alpha_0})$, since $\R_{\alpha_0}$ satisfies
the countable chain condition while $\theta$ is of uncountable cofinality,
\begin{equation}
  a=\|\dot r\in \pinf(t_\alpha)\text{ for cofinally many $\alpha$}\|\ne 0,
\end{equation}
where $\dot r$ is an $\R_{\alpha_0}$-name for the random real in $\C_\pi$. 

Let $G\subseteq\C_\pi\cong\pN$ be a $G_\delta$-set with $[G]=a$. 
Then supposing $r\in G$ be a random real over the ground model $V$,
in $V[r]$: $r\in\nobreak\pinf(t_\alpha)$ for all $\alpha<\theta$,
as $\pinf(t_\beta)\subseteq\pinf(t_\alpha)$ for all $\alpha<\beta$
by Proposition~\ref{prop:pinf-inc}.
Putting $x_\alpha=\{n:r(n)\subseteq t_\alpha(n)\}$, $x_\alpha$ is infinite and we have $x_\beta\subseteqfnt x_\alpha$
whenever $\alpha\le\beta$. Since $\theta<\p$ by~\cite[Theorem~1.9]{hirschorn2000towersm} that $\p$
is preserved under adding one random real, there exists an infinite
$y\subseteqfnt x_\alpha$ for all $\alpha$. Putting
\begin{equation}
  x=\bigcup_{n\in y}r(n),
\end{equation}
$x$ is thus an infinite pseudo-intersection of the tower. 

On the other hand, applying Corollary~\ref{cor:meas-one-h} to $t_{\alpha_0}$ implies
$r\in\intfin[\pi](b)$, and thus $x\in\intfin[\pi\circ e_y](b)$, for all $b\in H$.
In particular, $x\cap b$ is infinite for all $b\in H$.
Furthermore, we have the following result on a form of asymptotic density for random reals.

\begin{lem}
  Let $a\in\C_\pi$. Assume that $n\mapsto 2^{-|a(n)|}$ is summable.
    \label{lem:asymptotic}
  If $r\in\C_\pi$ is a random real with respect to the partition measure for\/ $a$, then
  \begin{equation}
    \label{eq:16}
    \lim_{n\to\infty}\frac{|r(n)|}{|a(n)|}=\frac12.
  \end{equation}
\end{lem}
\begin{proof}
  It follows from the Bernstein inequality~(e.g.~\cite{bernstein1924modification}) that given $\varepsilon>0$,
  \begin{equation}
    \nu_{\tilde a(n)}\left(\left\|\left|\frac{|\dot
            r(n)|}{|a(n)|}-\frac12\right|>\varepsilon\right\|\right)
     =O\bigl(2^{-|a(n)|}\bigr)
  \end{equation}
  (see equation~\eqref{eq:4}). 
\end{proof}

The tower $T$ lies in the ground model and is of cardinality $|T|<\p$. Hence
\cite[Corollary 5.11]{hirschorn2000towersm} implies that
the existence of an infinite pseudo-intersection of $T$ in the forcing extension by one random real
is absolute, i.e.~such an $x$ can  be found in the ground model $V$.\footnote{This result is
  nonvacuous because its hypothesis is that $|T|\le\min\{\f,\b\}$, where $\f$ is a cardinal
  characteristic of $\pnfin$ (not discussed here) with $\p\le\f$.}
The proof of~\cite[Corollary~5.11]{hirschorn2000towersm} 
can be generalized to prove that if $H$ is a ground model set of cardinality $|H|<\b$, then the
existence in $V[r]$ of an infinite-pseudo intersection $x$ of $T$ with $x\cap b$ infinite for all
$b\in H$ is also absolute. Furthermore, if additionally the asymptotic density of $x$, as defined in
equation~\eqref{eq:16}, lies in some rational interval in $V[r]$, this fact is also absolute. 

The preceding establishes two things. For one, if we are always in the nonessential case when trying
to prove that $\d^*\ge\b$ (and thus $\p=\t$), then this is already handled by the theory just exposed
for one random real, since we shall already have in the ground model an $[x]\subseteqfnt [T]$ in $\co([H])$.

However, it also shows that we may assume that we end up in the essential case, which is thus truly
``essential''. For suppose we
have a filter-base $B$ of cardinality $|B|\ge\aleph_2$. In
refining $B$ to a tower, following the proof of Lemma~\ref{lem:refine-to-tower},
we recursively find $t_\alpha\subseteq b_\alpha$ in $\co(H)$ which
$\subseteqfnt$-extends $\{t_\xi:\nobreak\xi<\nobreak\alpha\}$. At each limit $\delta$ of uncountable cofinality,
assume by way of contradiction that we end up in the nonessential case, as witnessed by the
measure $\mu_{\xi_\delta}$ where $\xi_\delta<\delta$. Then
applying our observations to obtain $t_\delta$,
it will moreover satisfy $\lim_{n\to\infty}\frac{|t_\delta(n)|}{|t_{\xi_\delta}(n)|}=\frac12$
by Lemma~\ref{lem:asymptotic}.

After refining for $\omega_1^2$ steps, by Pressing Down, there exists a stationary
$S\subseteq\omega_1$ and $\xi$ such that $\xi_{\omega_1\cdot\alpha}=\xi$ for all $\alpha\in S$.
Put $\delta=\omega_1\cdot\alpha<\omega_1\cdot\alpha'=\delta'$ for some $\alpha<\alpha'$ in $S$.
The fact that the limit of the cardinality ratio is $\frac12$ in particular implies that
\begin{equation}
  \mu_{\xi_\delta}\bigl(\pinf(t_\delta)\bigr)=0.
\end{equation}
However, by Proposition~\ref{prop:pinf-inc}, this entails that
$\mu_{\xi_\delta}\bigl(\pinf(t_\alpha)\bigr)=0$ for all $\delta\le\alpha<\delta'$, contradicting
that $\xi_\delta=\xi=\xi_{\delta'}$ witnesses the nonessential case for $\delta'$.

Now consider the essential case. Suppose that $x$ is a $\subseteqfnt$-lower bound of $T$.
Then $\pinf(x)\subseteq\pinf(t_\xi)$ for all $\xi<\theta$ by Proposition~\ref{prop:pinf-inc},
which means that $\mu_\xi\big(\pinf(x)\bigr)=0$ for all $\xi$.
Therefore $\nu_T\bigl(\pinf(x)\bigr)=0$. In conclusion, our measure theoretic strategy fails
for the ultrapower limit measure construction. 

\subsection{Limit as a Product}
\label{sec:limit-as-products}

Based on the observations in \S\ref{sec:limit-via-ultrafilter} on the ultrapower measure,
the only way our measure theoretic approach can possibly
work is if $\alpha$ of the tower enumeration varies over the partition index $n$.
Thus for a function $f:\N\to\theta$, we let $\mu_f=\mu_{f;T,\pi}$
be the product measure $\mu_f=\prod_{n=0}^\infty\nu_{I_\pi(n),t_{f(n)}}$ on~$\C_\pi$.
More generally we want to consider a product measure on a subset of
the indices. Supposing that $h\in\irr$ is strictly increasing,
$n\mapsto I_{\pi\circ h(n)}$ is a partition of
\begin{equation}
\N_{\pi\circ h}=\bigcup_{n=0}^\infty I_{\pi\circ h(n)}.
\end{equation}
Similarly as before, we can identify $\power(\N_{\pi\circ
  h})$ with $\C_{\pi\circ h}$.
The product measure $\mu_{f,h}=\mu_{f,h;T,\pi}$ on
$\C_{\pi\circ h}\cong\power(\N_{\pi\circ h})$ is then given by
\begin{equation}
  \label{eq:2}
  \mu_{f,h}=\prod_{n=0}^\infty\nu_{I_{\pi\circ h(n)},t_{f(n)}}
  =\prod_{n=0}^\infty\nu_{\tilde t_{f(n)}\circ h(n)} .
\end{equation}

\begin{prop}
  \label{prop:partition-meas}
  For every $n$,
  \begin{equation}
    \mu_{f,h}(\<U\>_{\pi\circ h})=\mu_{f(n)}(\<U\>_{\pi})
    \espc\textup{for all $U\subseteq \power(I_{\pi\circ h(n)})$.}
  \end{equation}
\end{prop}

As the partition measures are in some sense monotonic with respect to the enumeration of the tower
(e.g.~Proposition~\ref{prop:decr-meas}), in light of Proposition~\ref{prop:partition-meas},
we can view this product as a limit \emph{provided} the
range of $f$ is cofinal in $\theta$. Hence the need for collapsing cardinals, namely $\theta$ to a
countable ordinal, in order to obtain a measure as a limit of the tower.


While $\ideal_\theta^+$ does force a cofinal map from $\N$ into $\theta$, as follows from considering
an Ulam Matrix on $\theta$ (see e.g.~\cite[p.~79]{kunen83}), such a mapping can be forced in a more controlled
manner using stationarity combined with a ladder for climbing $\theta$.
Let $\ell$ be an $\omega$-ladder system on $\theta$ (referred to simply as a \emph{ladder system}),
i.e.~for each $\delta\in\limo\theta$, $\ell_\delta:\omega\to\delta$ is a strictly increasing cofinal
map. Then given a stationary
$S\subseteq\limo\theta$, there exists a stationary subset $S'\subseteq S$ such that
$\ell_\delta(0)$ is always equal to $\alpha_0$ for~$\delta\in S'$, and then we can find a
$S''\subseteq S'$ with $\ell_\delta(1)$ constant on $\delta\in S''$, and so on~\dots, as is
formalized in the sequel. 

\begin{definition}
Define $\Delta(\ell):\seq\theta\to\power(\limo\theta)$ by
\begin{equation}
  \Delta(s)=\Delta(s;\ell)=\bigl\{\delta\in\limo\theta:
  \ell_\delta(n)=s(n)\text{ for all $n<|s|$}\bigr\}.
\end{equation}
For a subfamily $\Scal\subseteq\power(\theta)$, 
let $Q(\Scal)=Q(\Scal;\ell)$ be the subtree of all $s\in\seq\theta$ where
\begin{equation}
  \Delta(s;\ell)\in\upcl{\Scal} 
\end{equation}
(i.e.~$\Delta(s)$ has a subset in $\Scal$).
\end{definition}

\begin{prop}
  \label{prop:qs-inc}
  Members $s$ of\/ $Q(\Scal)$ are strictly increasing\textup: $s(n)<s(n+1)$.
\end{prop}
\begin{proof}
  By the corresponding property of ladder systems.
\end{proof}

\begin{prop}
\label{prop:Q-inc}
  If $\Scal\subseteq\T$ then $Q(\Scal)\subseteq Q(\T)$.  
\end{prop}

\begin{prop}
  \label{prop:delta}
  Let $s,t\in\seq\theta$. Then
  \begin{myenumerate}
  \item $s\sqsubseteq t$ implies $\Delta(s)\supseteq\Delta(t)$,
  \item if $s$ and $t$ are incompatible then $\Delta(s)\cap\Delta(t)=\emptyset$,
  \end{myenumerate}
  i.e.~$\Delta$ is a \emph{pre-embedding} of the partial orderings $(\seq\theta,\sqsupseteq)$ into
  $(\power(\limo\theta),\subseteq)$.
\end{prop}

Henceforth, our focus in on the subfamily $\nso\theta^+$ of $\power(\limo\theta)$, where
$\nso\theta$ denotes the ideal of nonstationary subsets of $\limo\theta$.

\begin{definition}
  When we say that $\Scal\subseteq\nso\theta^+$ is
  \emph{downwards closed} we mean as a sub-partial order under the
  subset relation: If $T\in\Scal$ and $S\subseteq T$ is in
  $\nso\theta^+$ then $S\in\Scal$.
  We write $\downcl{\Scal}$ for the downwards closure of $\Scal$ in $\nso\theta^+$.  
\end{definition}

\begin{prop}
  \label{prop:downcl}
  For $S\in\nso\theta^+$, $s\in Q(\downcl{\{S\}})$ iff\/ $\Delta(s)\cap S\in\nso\theta^+$.
\end{prop}

The completion of $\nso\theta^+$ as a forcing notion is the Boolean algebra $\power(\limo\theta)\div\nso\theta$, 
and a downwards closed subfamily $\Scal\subseteq\nso\theta^+$ corresponds to this Boolean algebra below
the supremum of $\Scal$.

For each $\delta\in\limo\theta$ and $\alpha<\delta$,
find $\tilde g(\delta,\alpha)=\tilde g_{T,\pi,\ell}(\delta,\alpha):\N\to\N$ so that
\begin{equation}
  \label{eq:5}
  t_{\ell_\delta(n)}\setminus \pi\circ \tilde g(\delta,\alpha)(n)\subseteq
  t_\alpha\espc\text{for all but finitely many $n$},
\end{equation}
or in other words $t_{\ell_\delta(n)}\cap I_\pi(k)\subseteq t_\alpha$
for all $k\ge \tilde g(\delta,\alpha)(n)$. Then as $\theta<\b$,
there is a $g_T=g_{T,\pi,\ell}\in\irr$ such that
\begin{equation}
  \label{eq:18}
  \tilde g(\delta,\alpha)\lefnt g_T\espc\text{for all $\alpha<\delta\in\limo\theta$}.
\end{equation}

\begin{lem}
  \label{lem:dense}
  Suppose that $\Scal\subseteq\nso\theta^+$ is downwards closed, and $g_T\lefnt h$.
  Then for every\/~$q\in\nobreak Q(\Scal)$ and every finite $\vgamma\subseteq\theta$, 
  there exists $r\sqsupseteq q$ in $Q(\Scal)$ such that
  \begin{align}
    \label{eq:6}
    r(|r|-1)&\ge\max(\vgamma),\\
    \label{eq:7}
    t_{r(|r|-1)}\setminus \pi\circ h(|r|-1)&\subseteq t_\xi\espc\textup{for all $\xi\in\vgamma$}.
  \end{align}
\end{lem}
\begin{proof}
  Given $q\in Q(\Scal)$ and $\vgamma\subseteq\theta$, set $\beta=\max(\vgamma)$.
  Pick $S\subseteq\Delta(q)$ in $\Scal$.
  For every $\delta>\beta$ in $S$, there exists $n_\delta\ge|q|$ large enough so that $\ell_\delta(n_\delta)\ge\beta$
  and $t_{\ell_\delta(n_\delta)}\setminus \pi\circ h(n_\delta)\subseteq t_\xi$ for all $\xi\in\vgamma$ by equation~\eqref{eq:5}.
  Let $S'\subseteq S\setminus\beta+1$ be stationary with $n_\delta=\bar n$ for all
  $\delta\in S'$. By Pressing Down, there exists a stationary
  $S''\subseteq S'$ and $r\sqsupseteq q$ with $|r|=\bar n + 1$ and $\ell_\delta(n)=r(n)$
  for all $n=|q|,\dots,\bar n$, for all $\delta\in S''$.
  Thus $r\in Q(\Scal)$ is as required.
\end{proof}

\begin{cor}
  \label{cor:dense}
  For all $q\in\nobreak Q(\Scal)$, $\{r(|r|-1):r\in
  Q(\Scal)\text{, }r\sqsupseteq q\}$ is cofinal in $\theta$.
\end{cor}

\begin{rem}
  Corollary~\ref{cor:dense} in particular implies that $|\Q(\Scal)|=|\theta|$.
  For a stationary set $S\subseteq\limo\theta$,
  choosing an antichain $\A\subseteq\Q(\downcl{\{S\}})$ of cardinality $|\theta|$,
  $\{\Delta(s):s\in\nobreak \A\}$ partitions $S$ into $|\theta|$ many pairwise disjoint
  stationary subsets of $\Scal$ by Proposition~\ref{prop:delta}.
  This gives a proof Solovay's celebrated Splitting Theorem~\cite{solovay1971real} for the special
  case of stationary sets of countable cofinality ordinals.
\end{rem}

Now this determines an $\Scal$-name for a subset of $Q(\Scal)$ by
\begin{equation}
  \label{eq:8}
  X_{\dot\V}=X_{\dot\V;\ell}=\{s\in Q(\Scal;\ell):\Delta(s;\ell)\in\dot \V\}
\end{equation}
where $\dot\V$ is the (canonical) $\Scal$-name for the generic filter on $\Scal$.
Moreover, $X_{\dot\V}$ is branch through $\seq\theta$ which determines a cofinal mapping of $\N$ into $\theta$:

\begin{lem}
  If\/ $\Scal\subseteq\nso\theta^+$ is downwards closed then $\Scal$ forces that
  \begin{equation}
    f_{\dot \V}=\bigcup X_{\dot\V}:\N\to\theta
  \end{equation}
  is a strictly increasing function. Moreover,
  \begin{align}
    \label{eq:15}
    \Delta(q)&\forces q\sqsubset f_{\dot\V}\espc\textup{for all $q\in Q(\Scal)$,}\\
    \Scal&\forces\lim_{n\to\infty}f_{\dot \V}(n)=\theta.
  \end{align}
\end{lem}
\begin{proof}
  By Propositions~\ref{prop:qs-inc} and~\ref{prop:delta}, $f_{\dot\V}$ names a strictly increasing
  function. Given $k$ and $\beta<\theta$, let $S\in\Scal$ be arbitrary.
  It follows from Corollary~\ref{cor:dense} that there is an
  $r\in Q(\downcl{\{S\}})$ with $|r|\ge k$ and $r(|r|-1)\ge\beta$.
  Thus $S'=\Delta(r)\cap S\in\upcl\Scal$, and $S'\forces r\in X_{\dot\V}$,
  proving $\Scal$ forces that $|\dom(f_{\dot\V})|\ge k$ and $\sup_{n}f_{\dot
    \V}(n)\ge\beta$,
  and hence forces that the domain of $f_{\dot\V}$ is all of $\N$ and
  its range is cofinal.
\end{proof}

The ladder system serves another purpose to encode the interaction
between the tower representative $T$ and some other family $H$ of infinite subsets of
$\N$, in our case $[T]\subseteq\co([H])$.

\begin{definition}
  \label{def:satisfy}
  A property $\varphi(n,u,v)$ is said to be \emph{satisfied by\/ $(T,H)$ over $\pi$} if, when
  identifying $T$ and $H$ as subfamilies of $\C_\pi$, 
  \begin{equation}
    \label{eq:9}
    \varphi\bigl(n,t_\xi(n),b(n)\bigr)\text{ holds for all but finitely
      many $n\in\N$,}
  \end{equation}
  for all $\xi<\theta$ and $b\in H$. Then $g\in\irr$ is called an
  \emph{error correction for $(T,H)$ with respect to $\ell$} if whenever $g\lefnt h$,
  \begin{equation}
    \varphi\bigl(h(n),t_{\ell_\delta(n)}\circ h(n),b\circ h(n)\bigr)
    \text{ holds for all but finitely many $n$,}
  \end{equation}
  for all $\delta\in\limo\theta$ and all $b\in H$.
\end{definition}

\begin{prop}
  \label{prop:satisfy}
  Suppose $\theta,|H|<\b$.
  Then whenever $(T,H)$ satisfies a property $\varphi$ \tu(over some $\pi$\tu),
  it has an error correction with respect to $\ell$.
\end{prop}

\begin{lem}
  \label{lem:error-cor}
  Suppose that $(T, H)$ satisfies a property $\varphi$ with error correction~$g$ with respect to $\ell$.
  If\/ $g\lefnt h$,
  then in the forcing extension by\/ $\nso\theta^+$ with generic~$\V$,
  \begin{equation}
    \label{eq:20}
    \varphi\bigl(h(n), t_{f_{\V}(n)}\circ h(n), b\circ h(n)\bigr)
    \textup{ holds for all but finitely many $n$,} 
  \end{equation}
  for all\/ $b\in H$.
\end{lem}
\begin{proof}
  Let $b\in H$ be given.
  The definitions entail that each $\delta\in\limo\theta$ has an $n_\delta$ such that
  \begin{equation}
    \varphi\bigl(h(n),t_{\ell_\delta(n)}\circ h(n),b\circ h(n)\bigr)
    \text{ holds for all $n\ge n_\delta$.}
  \end{equation}
  Taking $S\in\nso\theta^+$, there exists $\bar n$ such that
  $S'=\{\delta\in S:n_\delta=\bar n\}$ is stationary.
  Now by Pressing Down using equation~\eqref{eq:15},
  \begin{equation}
    S'\forces \varphi\bigl(h(n),t_{f_{\dot \V}(n)}\circ h(n),b\circ h(n)\bigr)
      \text{ for all $n\ge \bar n$.}
  \end{equation}
\end{proof}

Naturally we are interested in properties of the form $\intprop(n,u,v;c)$ where
\begin{equation}
  \intprop(n,u,v;c)\Iff |u\cap v|\ge c(n)
\end{equation}
for all $n$.

\begin{cor}
  \label{cor:intfin}
  Let\/ $c\in\irr$.
  If\/ $(T, H)$ satisfies $\intprop(c)$ over $\pi$
  with error correction $g_{T,H}$ with respect to $\ell$, then whenever  $h\in\irr$ satisfies
  \begin{myenumerate}
  \item $h\gefnt g_{T,H}$ is strictly increasing,
  \item $\sum_{n=0}^\infty 2^{-c\circ h(n)}<\infty$,
  \end{myenumerate}
  every $\intfin[\pi\circ h](b)$ has measure one\tu:
  \begin{equation}
    \label{eq:12}
    \nso\theta^+\forces \mu_{f_{\dot\V},h}\bigl(\intfin[\pi\circ h](b)\bigr)=1
  \end{equation}
  for all\/ $b\in H$.
\end{cor}
\begin{proof}
  For every $b\in H$,
  whenever $\intprop\bigl(h(n), t_{f_{\dot\V}(n)}\circ h(n), b\circ h(n);c\bigr)$ holds (identifying
  $b$ and $t_\alpha$ as members of $\C_\pi$),
  $|t_{f_{\dot\V}(n)}\circ h(n)\cap b\circ h(n)|\ge c\circ h(n)$ implies that
  \begin{equation}
    \begin{split}
    \mu_{f_{\dot\V},h}\bigl(\bigl\<\intersects{I_{\pi\circ h(n)}}(b\circ h(n))\bigr\>\bigr)
    &=\nu_{I_{\pi\circ h(n)},t_{f_{\dot\V}(n)}}\bigl(\intersects{I_{\pi\circ h(n)}}(b\circ h(n))\bigr)\\
    &=1- 2^{-|t_{f_{\dot\V}(n)}\cap b\cap I_{\pi\circ h(n)}|}\\
    &\ge 1-2^{-c\circ h(n)}
    \end{split}
  \end{equation}
  by Lemma~\ref{lem:meas-H}. Now the result follows from the Cauchy criterion for products
  (cf.~proof of Corollary~\ref{cor:meas-one-h}).
\end{proof}

\begin{rem}
  A weaker but more natural property than satisfying $\intprop(c)$ is used in
  Corollary~\ref{cor:intfin}, namely that $\sum_{n=0}^\infty 2^{-|t_\alpha(n)\cap b(n)|}<\infty$. 
    One can generalize the definition of an error correction to apply to the latter
    property. However, if $(T,H)$ merely
  satisfies this weaker property then $\theta,|H|<\b$ is not sufficient to ensure the existence of
  an error correction (in fact, one would need $\theta,|H|<\add(\nulls)$ the additivity of Lebesgue measure).
\end{rem}

The mapping $\Delta$ of $Q(\Scal)$ into
$\Scal$ has been shown to determine a cofinal mapping  $f_{\V}:\N\to\nobreak\theta$ for $\Scal$-generic $\V$.
We need moreover for $Q(\Scal)$ to embed as a forcing notion into $\Scal$, i.e.~the 
corresponding branch $X_{\V}$ through $Q(\Scal)$ (c.f.~equation~\eqref{eq:8}) must intersect every dense subset of
$Q(\Scal)$. This is because the forcing notion
$Q(\Scal)$ has nice properties---its definition is absolute and its cardinality
$|Q(\Scal)|\le\theta$---that enable us the find an extension of $[T]$ which corresponds to subset of
$\C_\pi$ with positive measure.

On the other hand, we cannot simply force with $\Q(\Scal)$ instead
of $\Scal$ because equation~\eqref{eq:20} may fail in the former case, resulting in $\intfin(b)$
having measure zero. 

\begin{definition}
  $\Phi:P\to Q$ is a \emph{complete embedding} of the partial orders $(P,\le)$ into $(Q,\altle)$ if it a
  pre-embedding which preserves maximal antichains. That is, if $A\subseteq P$ is a maximal
  antichain then so is its image $\Phi[A]\subseteq Q$ (see e.g.~\cite{kunen83}).
\end{definition}

\begin{definition}
  A subset $D$ of a partial order $(P,\le)$ is said to be \emph{predense in~$S\subseteq P$}
  if every element $s$ of $S$ is compatible with some member $d$ of $D$. Thus $D$ is predense in $P$
  iff $D$ contains a maximum antichain, as usual. 
\end{definition}

The next result is dignified as a ``Theorem'' because to our knowledge it is new result on the
structure of stationary sets of ordinals of countable cofinality.

\begin{thm}
  \label{thm:stat-embedding}
  Let $\Scal\subseteq\nso\theta^+$ be downwards closed, and\/ $\ell$ a ladder system on~$\theta$.
  Then the image of every predense
  subset of $Q(\Scal;\ell)$ under $\Delta(\cdot;\ell)$ is predense in~$\Scal$.
  In particular, $\Delta(\cdot,\ell)$ is a complete embedding of\/ $Q(\nso\theta^+;\ell)$ into $\nso\theta^+$.
\end{thm}
\begin{proof}
  Each $K\subseteq\seq\theta$ forms a tree $\downcl{K}=\bigcup_{r\in K}\{q:q\sqsubseteq r\}$
  under the ordering~$\sqsubseteq$.
  When $A\subseteq\seq\theta$ is moreover an antichain (this is an equivalent notion for the tree
  and its reverse poset ordering),
  the associated tree $\downcl{A}$ has no infinite branches; in other words, turning the tree upside down
  it is still well-founded (because a branch through $\seq\theta$ can
  only intersect at most one element of $A$).

  The classical notion of the rank of a well-founded tree $(T,\le)$ is defined by
  \begin{equation}
    \rank_T(q)=
    \begin{cases}
      \hfill 0&\text{if $q$ is a terminal node,}\\
      \displaystyle\sup_{r\in\succ_T(q)}\rank_T(r)+1&\text{otherwise.}
    \end{cases}
  \end{equation}
  The rank of the tree $\rank(T)$ is $\rank_T$ of its root node (the notion of a tree used here has a unique
  root). 

  It is required to prove that $\Delta$ maps maximal antichains $A\subseteq Q(\Scal)$ to subsets
  that are predense in $\Scal$. 
  The proof is by induction on the rank of $\downcl{A}$ for all $\Scal$ simultaneously.
  In the base case, trivially $\Delta(\emptyset;\ell)=\limo\theta\supseteq S$ for all $S\in\Scal$. 

  For the induction step suppose that $\downcl{A}$ has rank $\gamma>0$.
  First observe that $\{\Delta(r;\ell):r\in\succ_{\downcl A}(\emptyset)\}$ is predense in $\Scal$:
  Given $S\in\Scal$, by Pressing Down there exists a stationary $S'\subseteq S$ and $\alpha$ such that
  $\ell_\delta(0)=\alpha$ for all $\delta\in S'$. Thus $\<\alpha\>\in Q(\Scal)$, and by maximality
  there exists $a\in A$ compatible with $\<\alpha\>$, i.e.~$\<\alpha\>\in\succ_{\downcl{A}}(\emptyset)$.
  
  For each $r\in Q(\Scal)$, let $A_r$ be the set of all $a\in A$ with~$a\sqsupseteq r$,
  and put $\Scal_r=\downcl\{\Delta(r;\ell)\}$.
  Let $\ell'$ be the ladder system $\ell$ shifted by one,
  i.e.~$\ell'_\delta=\sigma(\ell_\delta)$ for all $\delta$, where $\sigma(x)(n)=x(n+1)$ is the shift-operator.
  Then
  \begin{equation}
    \label{eq:14}
    \Delta(q;\ell)=\Delta(\sigma(q);\ell')\cap \Delta(\<q(0)\>;\ell)
    \espc\text{for all $q\ne\emptyset$ in $\seq\theta$.}
  \end{equation}
  $a\in A\setminus A_r$ implies $a\perp Q(\Scal_r)$ by Proposition~\ref{prop:downcl}. Hence, $A_r$
  is a maximal antichain of $Q(\Scal_r)$ by the maximality of $A$,
  and thus $\sigma[A_r]$ is a maximal antichain of  $Q(\Scal_r;\ell')$ by equation~\eqref{eq:14}.
    
  Now as $\sigma:\downcl{A_r}\setminus\{\emptyset\}\to\sigma[\downcl{A_r}]$
  is a tree isomorphism mapping $r$ to $\emptyset$,\linebreak
  $\rank(\sigma[\downcl{A_r}])=\rank_{\downcl{A_r}}(r)<\rank(\downcl A_r)\le\rank(\downcl A)=\gamma$.
  By induction,\linebreak $\{\Delta(\sigma(a);\ell'):a\in A_r\}$ is predense in $\Scal_r$, and hence so is
  $\{\Delta(a;\ell):a\in A_r\}$. Thus from the earlier
  observation that the union of the $\Scal_r$'s is predense, the image of $A$ under
  $\Delta(\cdot;\ell)$ is predense in $\Scal$.
\end{proof}

Following is the standard poset for forcing a common extension of a filter-base~$F$ in $\pnfin$,
which shall be applied to our tower $F=T$.

\begin{definition}
  \label{def:forcing-std}
  Given a family $F$ of subsets of $\N$, enumerated as $(t_\xi:\xi<\theta)$, let $P_F$ denote the partial
  ordering of all pairs $p=(s_p,\vgamma_p)$ where
  \begin{myenumerate}
  \item $s_p\subseteq\N$ is finite,
  \item $\vgamma_p\subseteq\theta$ is finite,
  \end{myenumerate}
  ordered by $q\le p$ if
  \begin{myenumerate}
  \setcounter{enumi}{2}
  \item $s_q\supseteq s_p$ and $\vgamma_q\supseteq\vgamma_p$,
  \item $s_q\setminus s_p\subseteq t_\xi$ for all $\xi\in\vgamma_p$.
  \end{myenumerate}
\end{definition}

Then a filter $G\subseteq P_F$ defines
\begin{equation}
  x_{G}=\bigcup_{p\in G}s_p\subseteq\N.
\end{equation}

\begin{prop}
  Whenever $G$ intersects the dense set
  \begin{equation}
    E_\xi=\{p\in P_F:\xi\in\vgamma_p\}\espc\textup{($\xi<\theta$)},
  \end{equation}
  we have $x_G\subseteqfnt t_\xi$.
\end{prop}

\begin{prop}
  $P_F$ is a $\sigma$-centered partial ordering. 
\end{prop}

\begin{definition}
  \label{def:dense}
  Let $h\in\irr$ be strictly increasing.
  For each $q\in Q(\Scal)$,
  define $D_q(\Scal)=D_q(\Scal;\ell,h,T,\pi)$ to be the set of all $p\in P_T$ for which
  there exists $r\sqsupseteq q$ in~$Q(\Scal)$ with 
  \begin{equation}
    \label{eq:3}
    \mu_{r(|r|-1)}\bigl(\bigl\<\power(s_p\cap I_\pi\circ h(|r|-1))\bigr\>\bigr)>1-\frac1{|r|}.
  \end{equation}
  Conversely, for $G\subseteq P_T$, define $R(G,\Scal;\ell,h,T,\pi)$ to be the set of all $r\in Q(\Scal)$ for which
  there exists $p\in G$ witnessing equation~\eqref{eq:3} for $r$.
\end{definition}

\begin{lem}
  \label{lem:density}
  Let $g_T\lefnt h$. Then
  $D_q(\Scal;h)$ is dense for all\/~$q\in Q(\Scal)$.
\end{lem}
\begin{proof}
  Take $q\in Q(\Scal)$.
  Given $p\in P_T$, by Lemma~\ref{lem:dense}, there exists $r\sqsupseteq q$ in
  $Q(\Scal)$ satisfying equations~(\ref{eq:6}) and~(\ref{eq:7}) for $\vgamma=\vgamma_p$.
  Put
  \begin{equation}
    s_{p'}=s_p\cup \bigl(I_\pi\circ h(|r|-1)\cap t_{r(|r|-1)}\bigr)=s_p\cup t_{r(|r|-1)}\circ h(|r|-1).
  \end{equation}
  Then $s_{p'}\setminus s_p\subseteq t_\xi$ for all $\xi\in\vgamma_p$,
  and therefore $p'=(s_{p'},\vgamma_p)\le p$. Since $s_{p'}\supseteq t_{r(|r|-1)}\circ h(|r|-1)$ and
  $\mu_{r(|r|-1)}(\<\power(t_{r(|r|-1)}\circ h(|r|-1))\>)=1$, $r$
  witnesses equation~(\ref{eq:3}) for $p'$, proving it is in $D_q(\Scal)$.  
\end{proof}

\begin{thm}[$\theta<\p$]
  \label{thm:pseudo-intersection}
  Suppose $h\in\irr$ is strictly increasing with $g_T\lefnt h$.
  Then there exists $y\subseteq\N$ such that $y\subseteqfnt T$ and
  $\nso\theta^+$ forces that
  \begin{equation}
    \label{eq:13}
    \limsup_{n\to \infty}\mu_{f_{\dot\V}(n)}\bigl(\bigl<\power(y\circ h(n))\bigr\>_\pi\bigr)=1.
  \end{equation}
\end{thm}
\begin{proof}
  Bell's Theorem~\cite{bell_combinatorial_1981} yields the existence of a filter $G\subseteq P_T$ intersecting
  \begin{myenumerate}
  \item\label{item:7} $E_\xi$ for all $\xi<\theta$,
  \item\label{item:5} $D_q(\nso\theta^+;h)$ for all $q\in Q(\nso\theta^+)$,
  \end{myenumerate}
  by Lemma~\ref{lem:density}, as $|Q(\nso\theta^+)|\le\theta<\p$.
  Condition~\ref{item:5} entails that $R(G,\nso\theta^+;h)$ is dense in $Q(\nso\theta^+)$.
  Thus for each $k$, we can find $B_k\subseteq R(G,\nso\theta^+;h)$ where
  \begin{myenumerate}
    \setcounter{enumi}{2}
  \item $B_k$ is a maximal antichain in $Q(\nso\theta^+)$,
  \item\label{item:4} $|r|\ge k$ for all $r\in B_k$.
  \end{myenumerate}
  Invoking Theorem~\ref{thm:stat-embedding},
  \begin{equation}
    \Delta[B_k]\text{ is a maximal antichain in }\nso\theta^+\espc\text{for all $k$.}
  \end{equation}

  By condition~\ref{item:7}, $y=x_G\subseteqfnt t_\xi$ for all $\xi<\theta$. Given
  $S\in\nso\theta^+$ and $k\in\N$, it now suffices to find a stationary $S'\subseteq S$ and $n\ge k$
  satisfying
  \begin{equation}
    \label{eq:10}
    S'\forces\mu_{f_{\dot\V}(n)}\bigl(\bigl\<\power(x_{G}\circ h(n))\bigr\>\bigr)>1-\frac{1}{k+1}.
  \end{equation}
  By the maximality of $\Delta[B_{k+1}]$ there exists $r\in B_{k+1}$ with $S'=\Delta(r)\cap S\in\nso\theta^+$.
  Since $r\in R(G,\nso\theta^+;h)$, there exists $p\in G$ verifying equation~\eqref{eq:3} for $r$.
  Setting $n=|r|-1\ge k$, since $s_p\cap I_{\pi}\circ h(n)\subseteq x_{G}\circ h(n)$
  and $\Delta(r)\forces f_{\dot\V}(n)=r(n)$, equation~\eqref{eq:3} confirms
  equation~\eqref{eq:10}.   
\end{proof}

\begin{cor}[$\theta<\p$]
  \label{cor:thick-pseudo}
  There exists a $\subseteqfnt$-extension $y$ of\/ $T$ such that
  \begin{equation}
    \nso\theta^+\forces \mu_{f_{\dot\V},h}\bigl(\pinf[\pi\circ h](y)\bigr)=1.
  \end{equation}
\end{cor}
\begin{proof}
  Let $y$ be as in Theorem~\ref{thm:pseudo-intersection}.
  Then the result follows from equation~\eqref{eq:13} since
  $\mu_{f_{\dot\V},h}\bigl(\bigcup_{n=k}^\infty\<\power(y\circ h(n))\>_{\pi\circ
    h}\bigr)\ge\limsup_{n\to\infty}\mu_{f_{\dot\V},h}\bigl(\bigl\<\power(y\circ
  h(n))\bigr\>_{\pi\circ h}\bigr)$ (see Proposition~\ref{prop:partition-meas}).
\end{proof}

\subsubsection{Application to Thick Lower Bounds}
\label{sec:appl-thick-lower}

In terms of our proposed strategy, the measure is $\mu_{f_\V,h}$ and the transformations are
$\varphi_T(x)=\pinf(x)$ and $\varphi_H(x)=\intfin(x)$.
Proposition~\ref{prop:p-int-h} in particular implies that
whenever $\varphi_T(x)\cap\varphi_H(b)$ has positive
$\mu_{f_\V,h}$-measure, $x\cap b$ is infinite. Full details are provided:

\begin{proof}[Proof of Theorem~\textup{\ref{thm:main}}]
  Let $[F]$ be a filter-base in $\co([H])$ of cardinality $|F|<\p$
  where $H$ is a subfamily of $\pN$ of
  cardinality $|H|<\b$. The task is to find a lower bound of $[F]$ in $\co([H])$.
  Using induction on the cardinality of $F$ and proceeding as in the proof of
  Lemma~\ref{lem:refine-to-tower} we may assume that $F=T$ is a tower with order-type $\theta<\p$ an
  infinite cardinal.

  Let $t_\alpha$ ($\alpha<\theta$) enumerate $T$ with respect to its
  well-ordering. In the case where~$\theta$ has
  countable cofinality, the result follows immediately from Theorem~\ref{thm:d} as $\b\le\d$.

  Now assume that $\theta$ is of uncountable cofinality. By Lemma~\ref{lem:summable-partition},
  there exists a partition
  function $\pi:\N\to\N$ such that
  \begin{equation}
    \label{eq:21}
    |t_\alpha\cap b\cap I_\pi(n)|\ge n\espc\text{for all but finitely many $n$,}
  \end{equation}
  for all $\alpha<\theta$ and $b\in H$. In other words, $\varphi_{\cap}(c)$ is satisfied by $(T,H)$
  over $\pi$, where $c$ is the identity function on $\N$.

  Let $\ell$ be a ladder system on $\theta$. By Proposition~\ref{prop:satisfy}, there exists an error correction
  $h\in\irr$ for $\varphi_{\cap}(c)$ with respect to $\ell$.
  Let $\V$ be a $\nso\theta^+$-generic ultrafilter on $\theta$. In~$V[\V]$: Corollary~\ref{cor:intfin} thereby gives us
  \begin{equation}
    \mu_{f_{\V},h}\bigl(\intfin[\pi\circ h](b)\bigr)=1\espc\text{for all $b\in H$}.
  \end{equation}
  And Corollary~\ref{cor:thick-pseudo} yields a common $\subseteqfnt$-extension $x$ of $T$, in $V$, satisfying
  \begin{equation}
    \mu_{f_{\V},h}\bigl(\pinf[\pi\circ h](x)\bigr)=1.
  \end{equation}
  Since two measure one sets have a measure one, and in particular nonempty, intersection,
  $\pinf[\pi\circ h](x)\cap\intfin[\pi\circ h](b)\ne\emptyset$ for all $b\in H$. Therefore, for
  every $b\in H$, by Proposition~\ref{prop:p-int-h},
  \begin{equation}
    V[\V]\models x\cap b\text{ is infinite}.
  \end{equation}
  By the absoluteness of $\Delta_0$ statements (c.f.~\cite[Ch.~IV]{kunen83}) for transitive models,
  \begin{equation}
    V\models x\cap b\text{ is infinite}. 
  \end{equation}
  This proves that $[x]$ is a member of $\co([H])$. 
\end{proof}

\begin{rem}
  The astute reader may have noticed the proof of Lemma~\ref{lem:density} establishes that the smaller
  $D'_q(\Scal; h)\subseteq D_q(\Scal; h)$ is dense, where $D'_q(\Scal; h)$ is the set of all $p\in P_T$ for
  which there exists $r\sqsupseteq q$ in $\Q(\Scal)$ with 
  \begin{equation}
    s_p\supseteq t_{r(|r|-1)}\circ h(|r|-1).
  \end{equation}
  One could use this observation to extract a proof of Theorem~\ref{thm:main} from the one
  presented here, that does not involve any measure. However, we have no interest in doing so.

  Such a proof would not add anything new, even if somewhat abridged.
  It would, however, hide the measure theoretic approach which served as the guiding
  light for finding the solution, which we believe to be just as important as the combinatorial result itself.
\end{rem}

\subsection{Ultraproduct Measure}
\label{sec:ultraproduct-measure}

An ultraproduct measure can be utilized as an alternative to the product measure $\mu_{\dot \V,h}$ of
Section~\ref{sec:limit-as-products}. Throughout this section, $\V$ is a $\nso\kappa^+$-generic
ultrafilter over the ground model $V$, and $\U$ will denote an ultrafilter on $\N$ over $V[\V]$.

For $h:\N\to\N$, the ultraproduct measure $\tau_{\U,f,h}=\tau_{\U,f,h;T,\pi}$ on $\C_\pi\div\U$ is then given by
\begin{equation}
  \tau_{\U,f,h}=\lim_{n\to\U}\nu_{I_\pi\circ h(n),t_{f(n)}}=\lim_{n\to\U}\nu_{\tilde t_{f(n)}\circ h(n)}.
\end{equation}
This can be compared with equation~\eqref{eq:2}.

Already we have found $\pi:\N\to\N$ satisfying equation~\eqref{eq:21}. In the construction of the
product measure, nonatomicity was unnecessary. However, it is easily obtained and is necessary for
the present argument. For each $\delta$, find $\tilde g(\delta):\N\to\N$ such
that
\begin{equation}
  \bigl|\tilde t_{\ell_\delta(n)}(k)\bigr|\ge n\espc\text{for all $k\ge\tilde g(\delta)(n)$,} 
\end{equation}
for all $n$.
Then there exists $\tilde g_T=\tilde g_{T;\pi,\ell}:\N\to\N$ with $\tilde g(\delta)\lefnt \tilde
g_T$ for all $\delta\in\limo\theta$.

The following can be compared with Proposition~\ref{sec:meas-nonatomic}.

\begin{lem}
  Whenever $\tilde g_T\lefnt h$,   $\tau_{\U,f_{\V},h}$ is nonatomic.
\end{lem}
\begin{proof}
  This is so because there exists an $S\in\V$ and $k$ with $\tilde g(\delta)(n)\le h(n)$
  for all~$n\ge k$, for all $\delta\in S$. 
\end{proof}

Then every $x\in\C_\pi$ is mapped by $\Q$ to a clopen subset of the ultraproduct
\begin{equation}
  \Q(x)=\prod_{n=0}^\infty\power(x)\div\U.
\end{equation}
Now our strategy has symmetrical $\varphi_T=\varphi_H=\Q$. 

\begin{lem}
  Let $\tilde g\lefnt h$. Then
  $\tau_{\U,f_\V,h}\bigl(\Q(x)\cap\Q(y)\bigr)>0$  implies $x(n)$ intersects~$y(n)$ for\/ $\U$-many $n$. 
\end{lem}
\begin{proof}
  Since $\tau_{\U,f_\V,h}$ is nonatomic,
  when $\Q(x)\cap\Q(y)$ has positive measure there exists $z\in\Q(x)\cap\Q(y)$ different from
  $\emptyset\div\U$, i.e.~$z(n)\ne\emptyset$ for $\U$-many $n$. 
\end{proof}

Analogously to Corollary~\ref{cor:intfin} we have the following consequence of Lemma~\ref{lem:error-cor}.

\begin{cor}
  \label{cor-ultrmeas-H}
  For any\/ $c\in\irr$, if $(T,H)$ satisfies $\varphi_{\cap}(c)$ over $\pi$ with error correction
  $g_{T,H}$ with respect to\/ $\ell$, then whenever $h\gefnt g_{T,H}$ is strictly increasing,
  \begin{equation}
    \tau_{\U,f_{\V},h}\bigl(\Q_\pi(b)\bigr)\ge1-2^{-\lim_{n\to\U}c(n)}
    \espc\textup{for all $b\in H$}.
  \end{equation}
\end{cor}

And analogously to Corollary~\ref{cor:thick-pseudo} we have the following immediate consequence of
Theorem~\ref{thm:pseudo-intersection}. 

\begin{cor}[$\theta<\p$]
  \label{cor:ultrameas-int}
  If\/ $h\gefnt g_T$ is strictly increasing then there exists $y\subseteqfnt T$ such that
  \begin{equation}
    \tau_{\U,f_{\V},h}\bigl(\Q_\pi(y)\bigr)=1.
  \end{equation}
\end{cor}

Corollaries~\ref{cor-ultrmeas-H} and~\ref{cor:ultrameas-int} give an alternative proof of
Theorem~\ref{thm:main}. The same proof as before goes through with the
ultraproduct measure $\tau_{\U,f_{\V},h}$ used in place of the
product measure $\mu_{f_\V,h}$, and $\pinf[\pi\circ h]$ and $\intfin[\pi\circ h]$ both replaced
with $\Q=\Q_\pi$. 
\section{Gaps in \texorpdfstring{$\irr\div\U$}{the irrational ultrapower}}
\label{sec:gaps}

Recall that a \emph{pre-gap} of type $(\psi,\tau)$ in some partial ordering $(P,\le)$, where $\psi$
and~$\tau$ are order-types, is a pair suborders $(A,B)$ where the order-types of $A$ and~$B$
are $\psi$ and $\tau$, respectively, and $a< b$ for all $a\in A$ and $b\in B$. It is further a
\emph{$(\psi,\tau)$-gap} if it cannot be \emph{interpolated}, i.e.~there is no $c$ in the partial order
with $a< c< b$ for all $a\in A$ and $b\in B$.

For an order type $\tau$, $\tau^*$ is the standard notation for the reversed order. Our choice of
symbol $\d^*$ is unfortunately inconsistent with said notation. In a Boolean algebra such as $\pnfin$,
$(\psi,\tau^*)$-gaps are transformed into $(\tau,\psi^*)$-gaps by taking complements.

$\t$ is the smallest ordinal $\kappa$ for which there exists a $(1,\kappa^*)$-gap in $\pnfin$ by its very definition.
One the main objects of study in Rothberger's paper~\cite{rothberger1948some} are his so called
\emph{$\varOmega$-limits} which are precisely $(\omega_1,1)$-gaps.
We can express the fact that $\p=\t$ with the following Proposition.  

\begin{prop}
  $\p$ is the smallest ordinal $\kappa$ for which there exists a $(1,\kappa^*)$-gap in $\pnfin$.
\end{prop}

Investigations into gaps in $\irrfin$ date back to the second half of the nineteenth
century. In 1873, du Bois--Raymond~\cite{bois1873neue} proved that there are no $(\omega,1)$-gaps in $\irrfin$;
that is, no $\omega$-limits in Rothberger's terminology. The nonexistence of $(\omega,\omega^*)$-gaps in
$\irrfin$ was then established by Hadamard~\cite{hadamard1894caracteres}.

\begin{thm}[Hadamard, 1894]
  There are no $(\omega,\omega^*)$-gaps in $\irrfin$.
\end{thm}

Notice that this is in contrast to the rationals which evidently do have $\omega$-limits and $(\omega,\omega^*)$-gaps.

Gaps in $\irrfin$ are related to the cardinal $\b$. For example, as in the next Lemma and
Rothberger's Theorem~\ref{thm:rothberger-first} below.

\begin{lem}
  \label{lem:b-gap}
  $\b$ is the smallest ordinal $\kappa$ for which there exists a $(\kappa,0)$-gap in\linebreak
  $\irrfin$. 
\end{lem}
\begin{proof}
  Given $f_\xi\in\irr$ ($\xi<\kappa$) where $\kappa\le\b$, one can recursively choose $g_\xi\gefnt f_\xi$ that
  are well-ordered, i.e.~$g_\xi\lneqfnt g_\eta$ for all $\xi<\eta$. 
\end{proof}

One of the most striking early discoveries in set theory was Hausdorff's
construction~\cite{hausdorff1909graduierung} of a special combinatorial object of cardinality exactly $\aleph_1$.

\begin{thm}[Hausdorff, 1909]
  There exist gaps of type $(\omega_1,\omega_1^*)$ in $\irrfin$.
\end{thm}

Later in 1936 he introduced $(\omega_1,\omega_1^*)$-gaps in $\pnfin$, in~\cite{hausdorff1936summen}.
Rothberger built upon Hausdorff's investigations,
for example in~\cite{rothberger_sur_1941}, which eventually lead
to his questions about refining filter-bases to towers.

\begin{thm}[Rothberger, 1941]
  \label{thm:rothberger-first}
  $\b$ is the smallest ordinal $\kappa$ for which there exists an $(\kappa,\omega^*)$-gap in
  $\pnfin$, and similarly for\/~$\irrfin$.
  \footnote{Owing to the
    similarities between these two partial orders, the existence of a gap of type $(\psi,\tau)$ in
    $\irrfin$ is equivalent to the existence of a $(\psi,\tau)$-gap in $\pnfin$.
    This was established by Rothberger in~\cite{rothberger_sur_1941}}.
\end{thm}

If $\U$ is the generic ultrafilter of $\pnfin$ then $\U$ is an nonprincipal ultrafilter on~$\N$ in
the forcing extension $V[\U]$, as $\pnfin$ does not add new reals (i.e.~subsets of $\N$).\footnote{Compare
  with Lemma~\ref{lem:force-ultra-over} for $\power(\omega)\div\ideal_\omega\cong\pnfin$.}

Concerning gaps in $\irr\div\U$ for $\U$ a generic ultrafilter,
our understanding is that the following is the result on gaps in $\irr\div\U$ proved by
Malliaris---Shelah, and moreover that this was first proven for $\t$ instead of $\p$ in the process of proving
their equality. 

\begin{thm}[Malliaris--Shelah]
  \label{thm:shelah-gap}
  There are no $(\theta,\lambda^*)$-gaps for $\theta,\lambda<\p$, with at least one of them a
  limit,\footnote{To avoid the degenerate case of $(1,1)$-gaps.}
  in $\irr\div\U$ for $\U$ a generic ultrafilter.
\end{thm}
\begin{proof}[Sketch of Proof]
  Such a gap $(A,B)$ would be in the ground model as $|A|,|B|<\p\le\h$,
  where $\h$ is another cardinal characteristic of $\pnfin$, namely its distributivity number.
  Working with representatives $A,B\subseteq\Fin^+$,
  suppose $x\in\Fin^+$ forces that $(A,B)$ is a $(\theta,\lambda^*)$-pregap.
  Then the areas above the graphs $a\restriction x$ ($a\in A$) and the areas below the graphs
  $b\restriction x$ ($b\in B$) together form a filter-base, and hence has an infinite pseudo-intersection
  $y\subseteq\N\times\N$ with $d=\dom(y)\subseteq x$ and finite vertical sections.
  Thus $d$ is infinite and forces that $[n\mapsto \max\{k:(n,k)\in y\}]$ interpolates $(A,B)$ in $\irr\div\U$.
\end{proof}

There was no obvious application of our new result, that $\d^*\ge\b$, to gaps in $\irr\div\U$ for
$\U$ generic. Perhaps this is because the following question (open to our knowledge) has a positive
answer, in which case Theorem~\ref{thm:shelah-gap} is optimal.

\begin{question}
  Are there $(\p,\lambda^*)$-gaps \tu($\lambda\le\p$\tu) in $\irr\div\U$ for $\U$ generic?
\end{question}

\subsection{Universal Gaps}

While there are endless variations one could consider for the ultrapowers $\irr\div\U$, universality
properties are interesting in general: Let $\b_{\wh\U}$\footnote{This notation is to distinguish from
  $\b_\U$ which would naturally refer to unbounded families in $\irr\div\U$ for a specific $\U$. We
  do not consider $\b_\U$ here.}
  the smallest cardinality of a subfamily of
$A\subseteq\irr$ for which $A\div\U$ is unbounded in $\irr\div\U$ for every nonprincipal ultrafilter $\U$ on $\N$. 
The ultrafilter property entails that $\irr\div\U$ is totally ordered, and thus ``unbounded'' and
``cofinal'' are equivalent properties. Evidently:

\begin{lem}
  $\b\le\b_{\wh\U}=\d_{\wh\U}\le\d$.
\end{lem}

\begin{definition}
  A pair $(A,B)$ of subfamilies of $\irr$ is called a \emph{universal \tu(pre-\tu)gap for ultrapowers} if
  $(A\div\U, B\div\U)$ is a (pre-)gap in $\irr\div\U$ for every nonprincipal ultrafilter $\U$ on $\N$.

  $f\len\U g$ for every nonprincipal ultrafilter $\U$ on $\N$ is equivalent to
  $f\lefnt g$, and $f\lneqn\U g$ for all such $\U$ is equivalent to $f\lfnt g$.
  Hence a universal pregap satisfies $f\lfnt g$ for all $f\in A$ and $g\in B$, and its
  order-types are with respect to the $\lfnt$ ordering on $\irr$.
\end{definition}

Since an order-type is a partial ordering modulo order isomorphism, the following is just a
reformulation of the definitions.

\begin{prop}
  $\b_{\wh\U}$ is the smallest cardinality of a order-type $\tau$ for which there exists a universal
  $(\tau,0)$-gap in ultrapowers of $\irr$.
\end{prop}

\begin{lem}
  If $\b=\b_{\wh\U}$ then there exists a universal $(\b,0)$-gap for ultrapowers of\/~$\irr$.
\end{lem}
\begin{proof}
  A universally unbounded subfamily of $\irr$ can be dominated by a well-ordered subfamily of $\irr$
  when its cardinality is only $\b$.
\end{proof}

From the folklore, a \emph{scale in $(\irr,\lefnt)$} is a well-ordered cofinal subfamily. If $\b=\d$
then there exists a scale in $\irr$.  

\begin{lem}
  If there is a scale in $\irr$ of order-type $\kappa$ then 
    there exists a universal $(\kappa,0)$-gap for ultrapowers of\/ $\irr$. 
\end{lem}

Though we did not attempt to find a counterexample,
we would conjecture that universal $(\kappa,0)$-gaps need not exist.
In fact it is not clear that any universal gaps of well-ordered type exists at all.
To be noted, the following question could be asked about gaps in $\irr\div\U$ for $\U$ generic as well.

\begin{question}
  Is it consistent that there is no universal gap for ultrapowers of\/ $\irr$ of type $(\tau,\psi)$
  with both $\tau$ and $\psi$ well-ordered?
\end{question}

An order-type $\tau$ is said to be \emph{filtered} if every finite $F\subseteq\tau$ has an upper
bound in~$\tau$. 
The following theorem details the relevance of the cardinal $\d^*$ to universal gaps for ultrapowers
of $\irr$. 

\begin{thm}
  For\/~$\tau$ a filtered order-type of cardinality $|\tau|<\d^*$ and~$\lambda<\p$ a limit, there are no
  universal $(\tau,\lambda^*)$-gaps for ultrapowers of\/ $\irr$.
\end{thm}
\begin{proof}
  Suppose $A,B\subseteq\irr$ have order-types $\tau$ and $\lambda^*$ under $\lfnt$, respectively.
  For each $f\in A$, let $a_f=\{(n,k):k> f(n)\}$ be the area strictly above the graph of~$f$,
  and for each $g\in B$, let $b_g=\{(n,k):k< g(n)\}$ be the area strictly below the
  graph of $g$.

  For all $f\in A$ and $g\in B$, there is a $g'\lfnt g$ in $B$ since $\lambda$ is a limit,
  and $f\lfnt g'\lfnt g$ implies that $a_f\cap b_g$ is infinite; while the $b_g$
  ($g\in B$) are well-ordered under $\supsetfnt$ and thereby form a filter-base. 
  Therefore, $|A|<\d^*$ yields $c\subseteq\N\times\N$ with $c\subseteqfnt b_g$ for all
  $g\in B$ and $c\cap a_f$ infinite for all $f\in A$.

  All but finitely many elements of
  $\dom(c)$ have finite vertical sections, denoted as $x_c$.
  Letting $h\in\irr$ satisfy
  \begin{equation}
    h(n)=\max\{k:(n,k)\in c\}
  \end{equation}
  for all $n\in x_c$,   
  $h(n)< g(n)$ for all but finitely many $n\in x_c$ for all $g\in B$. 
  For every $f\in A$, 
  \begin{equation}
    y_f=\{n\in x_c:f(n)< h(n)\}
  \end{equation}
  is infinite. Every finite $\vgamma\subseteq A$ has a $\lfnt$-upper bound $f'\in A$. Thus
  $\bigcap_{f\in\vgamma}y_f\supseteqfnt y_{f'}$ is infinite, establishing that there exists a
  non-principal ultrafilter $\U\ni y_f$ for all~$f\in A$.  Now $f\lln\U h\lln\U g$ for all $f\in A$
  and $g\in B$.
\end{proof}

Applying Theorem~\ref{thm:main} yields the following Corollary.

\begin{cor}
  For all $\kappa<\b$ and limit $\lambda<\p$, there are no universal $(\kappa,\lambda^*)$-gaps for
  ultrapowers of\/~$\irr$. 
\end{cor}

\printbibliography
\end{document}